\documentclass[11pt]{article}
\usepackage[centertags]{amsmath}
\usepackage{amsfonts}
\usepackage{amssymb}
\usepackage{amsthm}
\usepackage{soul}
\usepackage{mathrsfs}
\usepackage{extarrows}
\usepackage[margin=1in]{geometry}
\usepackage{multirow}
\usepackage{array}
\usepackage{tocloft}
\usepackage{tabularx}
\usepackage{booktabs}
\usepackage{enumitem}
\usepackage{tikz-cd}
\usepackage{caption}
\usepackage{hyperref}

\newcommand{\vertiii}[1]{{\left\vert\kern-0.25ex\left\vert\kern-0.25ex\left\vert #1
    \right\vert\kern-0.25ex\right\vert\kern-0.25ex\right\vert}}

\theoremstyle{plain}
\newtheorem{thm}{Theorem}[section]

\newtheorem{lem}[thm]{Lemma}
\newtheorem{prop}[thm]{Proposition}

\theoremstyle{definition}
\newtheorem{construction}[thm]{Construction}
\newtheorem{defn}[thm]{Definition}
\newtheorem{rem}[thm]{Remark}
\newtheorem{exam}[thm]{Example}

\newcommand{\CC}{\mathbb{C}}
\newcommand{\DD}{\mathbb{D}}

\newcommand{\NN}{\mathbb{N}}

\newcommand{\RR}{\mathbb{R}}

\newcommand{\TT}{\mathbb{T}}

\newcommand{\ZZ}{\mathbb{Z}}

\newcommand{\BBB}{\mathcal{B}}
\newcommand{\CCC}{\mathcal{C}}

\newcommand{\HHH}{\mathcal{H}}
\newcommand{\III}{\mathcal{I}}

\newcommand{\LLL}{\mathcal{L}}

\newcommand{\SSS}{\mathcal{S}}

\newcommand{\la}{\langle}
\newcommand{\ra}{\rangle}
\newcommand{\intd}{\mathrm{d}}

\newcommand{\bz}{\bar{z}}
\newcommand{\bw}{\bar{w}}

\newcommand{\zbf}{\mathbf{z}}

\newcommand{\polyn}{\mathbb{C}[\mathbf{z}]}

\newcommand{\poly}{\mathbb{C}[z]}
\newcommand{\Tbf}{\mathbf{T}}

\newcommand{\Tr}{\mathrm{Tr}}

\newcommand{\SA}{\mathcal{SA}}
\newcommand{\SArm}{\mathrm{SA}}

\newcommand{\Han}{\operatorname{Han}}

\allowdisplaybreaks[4]

\makeatletter\@addtoreset{equation}{section} \makeatother
\title {On the Schur--Agler Norm}
\author{
Michael Hartz
\thanks{Fachrichtung Mathematik, Universit\"at des Saarlandes, 66123 Saarbr\"ucken, Germany. M.H. was partially supported by the Emmy Noether Program of the German Research Foundation (DFG Grant 466012782)},
Yi Wang
	\thanks{College of Mathematics and Statistics, Center of Mathematics, Chongqing University, 401331, Chongqing,  China, wang\_yi@cqu.edu.cn}
    }
\date{\today}

\setlist[enumerate,1]{label=(\arabic*)}
 \setlist[enumerate,2]{label=(\alph*)}
 \setlist[enumerate,3]{label=(\roman*)}

\begin{document}
\maketitle

\begin{abstract}
  We establish a new description of the Schur--Agler norm of a holomorphic function on the polydisc as the solution of a convex optimization
  problem. Consequences of this description are explored both from a theoretical and from a practical point of view.
  Firstly, we give unified proofs of the known facts that the Schur--Agler norm can be tested with diagonalizable or nilpotent matrix tuples, as well as a new proof of the existence of Agler decompositions. Secondly, we describe the predual of the Schur--Agler space as a space of analytic functions on the polydisc. Thirdly, we give a unified treatment of existing counterexamples of von Neumann's inequality in our framework, and exhibit several methods for constructing counterexamples. On the practical side, we explain how the Schur--Agler norm of a homogeneous polynomial
  can be numerically approximated using semidefinite programming.

~

\noindent{Keywords}: Schur--Agler norm, von Neumann's inequality, And\^{o}'s inequality, commuting contractions, Hankel operator, semidefinite programming

\noindent{MSC (2020)}: Primary 47A13; Secondary 47A20, 47B35, 90C22
	
	\end{abstract}

\section{Introduction}

In 1951, John von Neumann proved the following inequality in \cite{vN51}: suppose $\HHH$ is a Hilbert space and $T\in\BBB(\HHH)$ is a contraction, i.e., $\|T\|\leq1$, then for any analytic polynomial $p\in\poly$,
\[
\|p(T)\|\leq\|p\|_\infty.
\]
Here $\|p\|_\infty$ denotes the supremum norm of $p$ on the unit disk $\DD$. In 1963, Tsuyoshi And\^{o} proved the two variable analogue of von Neumann's inequality in \cite{Ando63}. That is, for a commuting pair $(T_1, T_2)$ of contractions on a Hilbert space and any analytic polynomial $p\in\CC[z_1, z_2]$,
\[
\|p(T_1, T_2)\|\leq\|p\|_\infty:=\|p\|_{\DD^2}.
\]
However, in 1973, N.\ Th.\ Varopoulos proved in \cite{Varo73} that the corresponding inequality for three or more commuting contractions fails. Explicit counterexamples were constructed, for example, in \cite{CrDa75}\cite{Dix76}\cite{Hol01}\cite{Varo74}. 
Efforts to generalize or to explain this phenomenon have provoked research in operator theory, operator algebras and function theory. 
The survey paper \cite{BaBe13} and the books \cite{AgMcPickInterp}\cite{AgMcYo20book}\cite{Paulsen_CBMaps}\cite{PisierBook01} contain a rich source of references on related topic. Also see the papers \cite{AgMc05}\cite{Hartz17}\cite{Hartz25}\cite{Knese16}\cite{Knese21}\cite{knese25} for some more recent developments.

A commuting tuple $\Tbf=(T_1,\cdots,T_d)$ of operators is said to consist of strict contractions if $\|T_i\|<1$ for each $i$.
For a holomorphic function $f$ on $\DD^d$, its Schur--Agler norm is 
\[
\|f\|_{\SArm}:=\sup\{\|f(\Tbf)\|~:~\Tbf=(T_1,\cdots,T_d) \text{ is a commuting tuple of strict contractions}\}.
\]
The function $f$ is said to be in the Schur--Agler space $\SA_d$ if $\|f\|_{\SArm}<\infty$.
It is elementary to show that $\|f\|_{\SArm}\geq\|f\|_\infty$. Thus $\SA_d\subseteq H^\infty(\DD^d)$. The classical von Neumann inequality and the And\^{o} inequality show that when $d=1, 2$, $\|f\|_{\SArm}=\|f\|_{\DD^d}$ and $\SA_d=H^\infty(\DD^d)$. For $d\geq3$, it is not known whether the von Neumann's inequality holds up to a constant. Equivalently, it is not known whether the two norms are equivalent, or whether the two spaces contain the same set of functions. Let
\[
C(d)=\sup\left\{\frac{\|p\|_{\SArm}}{\|p\|_\infty}~:~p\in\polyn, p\neq0\right\}=\sup\left\{\frac{\|f\|_{\SArm}}{\|f\|_\infty}~:~f\in\SA_d, f\neq0\right\}.
\]
(The equality can for instance be seen by using Fej\'er means, cf.\ Lemma \ref{lem:SA_approx} below.)
Thus, the question is whether $C(d)$ is finite for $d\geq3$. Let $P_{d,n}$ be the space of homogeneous polynomials in $d$ variables of degree $n$. When no confusion arises, we simply write $P_n$ for $P_{d,n}$. Define
\[
C(d,n)=\sup\left\{\frac{\|p\|_{\SArm}}{\|p\|_\infty}~:~p\in P_{d,n},~p\neq0\right\}.
\]
By adding a variable to make a polynomial homogeneous, one can show that
\[
C(d)\leq\sup_n C(d+1,n)\leq C(d+1),
\]
reducing many questions to the homogeneous case.
In \cite{Dix76}, P.\ G.\ Dixon gave upper and lower bounds for $C(d,n)$:
\[
d^{\frac{1}{2}[\frac{n-1}{2}]}\lesssim_n C(d,n)\lesssim_n d^{\frac{n}{2}-1}.
\]
Here, the notation $A \lesssim_n B$ means that there exists a constant $C$, depending only on $n$,
such that $A \le C B$.
In \cite{Hartz25}, the first author showed that for $d\geq3$,
\[
C(d,n)\lesssim_d\left(\log(n+1)\right)^{d-3}.
\]
In particular, $\sup_nC(3,n)<\infty$.

This paper is motivated by the following simple observations:
\begin{enumerate}
    \item to compute the Schur--Agler norm, it suffices to consider cyclic commuting tuples of strict contractions;
    \item a cyclic commuting tuple $\Tbf$, with a distinguished cyclic vector $\xi$ is, up to unitary equivalence, determined by the semi-inner product
    \[
    \la p, q\ra_{\Tbf,\xi}:=\la p(\Tbf)\xi, q(\Tbf)\xi\ra,\quad\forall p, q\in\polyn=\CC[z_1,\cdots,z_d].
    \]
    \item for strictly contractive cyclic commuting tuples $(\Tbf,\xi)$, the associated semi-inner product $\la\cdot,\cdot\ra_{\Tbf,\xi}$ uniquely determines a positive bounded operator $L_{\Tbf,\xi}$ on the Hardy space $H^2(\DD^d)$ via
    \[
    \la L_{\Tbf,\xi}p,q\ra_2=\la p, q\ra_{\Tbf,\xi},\quad\forall p, q\in\polyn.
    \]
    Let us define the following convex cone of positive operators. 
\[
\LLL_c=\left\{L\in\BBB(H^2(\DD^d))~:~L\geq0,~M_{z_i}^\ast LM_{z_i}\leq L,~i=1,\cdots,d\right\}.
\]
Then $L_{\Tbf,\xi}\in\LLL_c$. Conversely, every operator $L\in\LLL_c$ defines a commuting tuple $\mathbf{T}$ with cyclic vector $\xi$ such that
\[ 
    \langle p(\mathbf T) \xi, q(\mathbf T) \xi \rangle = \la L p ,q \ra_2.
\]
\end{enumerate}
We explain the above in more detail in the beginning of Section \ref{sec: conv descrip of SA norm}.
The observations above leads to the following description of the Schur--Agler norm as the solution
of a convex optimization problem.
\begin{lem}[Lemma \ref{lem: SA norm with Lc}]\label{lem: intro SA in terms of Lc}
For any $f\in H^2(\mathbb{D}^d)$,
\[
\|f\|_{\SArm}^2=\sup\left\{\frac{\la L f, f\ra_2}{\la L1,1\ra_2}~:~L\in\LLL_c, L \neq 0  \right\},
\]
where both sides of the equality are allowed to be infinite.
\end{lem}
For homogeneous $p\in P_{d,n}$, the description above reduces to the finite dimensional convex cone
\[
\LLL_c^{(n)}=\left\{L\in\LLL_c~:~L=\sum_{k=0}^nL_k,\text{ where }L_k\in\BBB(P_{d,k})\right\},
\]
see  Lemma \ref{lem: SA by graded}.

We remark that the operator tuples corresponding to $L\in\LLL_c^{(n)}$ are nilpotent of order $n+1$. For comparison,
we also give a parallel convex description of the supremum norm (Lemma \ref{lem:sup_norm_Lt}), where the cone $\LLL_c$ is replaced by the cone
\begin{align*}
  \LLL_t &=\left\{L\in\BBB(H^2(\DD^d))~:~L\geq0,~M_{z_i}^\ast LM_{z_i}=L,~i=1,\cdots,d\right\} \\
         &=\left\{\text{positive Toeplitz operators on }H^2(\DD^d)\right\}.
\end{align*}

Our convex descriptions are closely related to Agler's linear functional approach to cyclic operators and the Schur--Agler norm (cf. \cite{Agler82}\cite{AgMcYo20book}). They are also related to the approach in \cite{Wanglinearfunctional}.
To illustrate the usefulness of our framework,
we give new proofs of two existing results. First, we give a short and uniform proof of the fact
that when computing the Schur--Agler norm, it suffices to consider either nilpotent or diagonalizable matrices.
We write $\mathcal{C}_d$ for the set of cyclic commuting $d$-tuples of strictly contractive matrices (of arbitrary size).

\begin{prop}
  \label{prop:nilpotent-diagonal-intro}
For $f\in H^\infty(\mathbb{D}^d)$,
\begin{flalign*}
    \|f\|_{\SArm}=&\sup\left\{\|f(\Tbf)\|~:~\Tbf \in \mathcal{C}_d \text{ is jointly nilpotent}\right\}\\
    =&\sup\left\{\|f(\Tbf)\|~:~\Tbf \in \mathcal{C}_d \text{ is jointly diagonalizable}\right\}.
\end{flalign*}
\end{prop}
The fact that diagonalizable matrices suffice follows from a theorem of Agler \cite{Agler90decomposition};
something more general is proved in \cite{AgMcYo13}. The statement about nilpotent matrices, along with a discussion of the history, can be found in \cite{knese25}.

The key step in our proof of Proposition \ref{prop:nilpotent-diagonal-intro} is to show that for the computation of the Schur--Agler norm, it suffices to consider commuting contractive tuples that are similar to $M_z$ on the Hardy space (Proposition \ref{prop:similar_to_M_z}). The proof of this fact becomes very short in the $L$ operator language: in $\LLL_c$, one may approximate $L$ with $L+\epsilon I$. We also give a new proof for the existence of the Agler decomposition (Theorem \ref{thm: Agler decomp}). The key idea is to replace the pairing $\mathrm{Her}(\DD^d)\times\left(\mathrm{Her}(\DD^d)\right)^*$ (see Remark \ref{rem: two pairings}) with the pairing $\SSS^1(H^2(\DD^d))\times\BBB(H^2(\DD^d))$, and to recover a hereditary function as a kernel function of an operator in $\SSS^1(H^2(\DD^d))$.

The convex descriptions also allow us to realize the pre-dual of the Schur--Agler space as an analytic function space on $\DD^d$. For $g\in H^2(\DD^d)$, define
\[
\LLL_c(g)=\left\{L\in\LLL_c~:~L\geq g\otimes g\right\},\quad\text{and}\quad \|g\|_\ast=\inf\left\{\sqrt{\la L1,1\ra_2}~:~L\in\LLL_c(g)\right\}.
\]
We prove the following.
\begin{thm}[Theorem \ref{thm:dual_norm_general} (1)]
  \label{thm: intro dual_norm_general}
The dual Schur--Agler norm $\|\cdot\|_\ast$ is a norm, and for each  $f\in H^2(\mathbb{D}^d)$, we have
  \begin{equation*}
    \|f\|_{\SArm} = \sup \{ | \langle f,g \rangle_2|: g\in H^2(\mathbb{D}^d), \|g\|_* \leq 1 \}.
  \end{equation*}
  In fact, the map
  \begin{equation*}
    \SA_d \to (H^2(\mathbb{D}^d), \|\cdot\|_*)^*, \quad f \mapsto \langle \cdot,f \rangle_2,
  \end{equation*}
  is a conjugate linear isometric isomorphism.
\end{thm}
Correspondingly, the reduced convex description on homogeneous polynomials gives a reduced description of $\|q\|_\ast$ for $q\in P_{d,n}$. Also, $\|\cdot\|_{\SArm}$ and $\|\cdot\|_\ast$ are dual to each other when restricted to each $P_{d,n}$ (Theorem \ref{thm:dual_norm_general} (2)). The convex description of the supremum norm gives a realization of the predual of $H^\infty(\DD^d)$ (Proposition \ref{prop: dual sup norm}).
Let $(\SA_d)_*$ denote the completion of $(H^2(\mathbb{D}^d), \|\cdot\|_*)$. Then by Theorem \ref{thm: intro dual_norm_general}, this space is exactly the predual of $\SA_d$. In Proposition \ref{prop: dual SA basics}, we settle some basic properties for $(\SA_d)_*$. In particular, it is an analytic function space on $\DD^d$. We believe that the predual Schur--Agler space $(\SA_d)_*$ may open the door for new research.

With the $L$ operator framework, we are also able to recast the counterexamples constructed in \cite{CrDa75}\cite{Dix76}\cite{Hol01}\cite{Varo74}. In the Appendix, we give computations of their $L$ operators and give interpretations of their construction. One conclusion from these examples is that the Hankel operators play an important role in their construction. Let $g\in H^\infty(\DD^d)$ and let $\Gamma_g\in\BBB(H^2(\DD^d))$ be the small Hankel operator with symbol $g$. With the well-known identity $M_{z_i}^\ast\Gamma_g=\Gamma_gM_{z_i}$, it is easy to show that $\Gamma_g^\ast\Gamma_g\in\LLL_c$. This construction alone does not give us any counterexample of the von Neumann's inequality. However, for a homogeneous polynomial $q\in P_{d,n}$ and $0\leq k\leq n$, one may construct 
\[
L=\Gamma_q^\ast\Gamma_q\big|_{P_{d,\geq k}}+CP_{d,<k}.
\]
Here $C>0$ is an appropriately chosen constant, and we abuse notation by writing $P_{d,\geq k}$ (resp. $P_{d,<k}$) for both the space of polynomials of degree $\geq k$ (resp. $<k$) and the projection onto it. From this choice of $L$, we obtain two methods for constructing counterexamples (Method 1 and Method 2 in Section \ref{sec: methods of constructing L}). It turns out that the counterexamples in \cite{CrDa75}\cite{Dix76}\cite{Varo74} all match Method 2. The Holbrook example \cite{Hol01} is more delicate. But we recover the same lower bound with improved methods (Methods 3 and 4).

Our methods of constructing $L$ give upper bounds for the dual Schur--Agler norm $\|\cdot\|_\ast$, which results in lower bounds for the Schur--Agler norm. These lower bounds admit a fairly explicit description in terms of weak products.
Define
\[
\|p\|_{P_{d,k}\odot P_{d,l}}=\inf\left\{\sum_{i=1}^m\|f_i\|_2\cdot\|g_i\|_2~:~f_i\in P_{d,k},~g_i\in P_{d,l},\text{ and }p=\sum_{i=1}^m f_ig_i\right\},\quad\forall p\in P_{d,k+l},
\]
and
\[
  \|p\|_{Z_d\odot P_{d,k}\odot P_{d,l}}=\inf\left\{\sum_{i=1}^d\|f_i\|_{P_{d,k}\odot P_{d,l}}~:~f_i\in P_{d,k+l},\text{ and }p=\sum_{i=1}^d z_if_i\right\},\quad\forall p\in P_{d,k+l+1}.
\]
Again, we will omit the subscripts ``$d$'' when no confusion is caused.
For any $p\in P_{n}$, let us also define
\begin{equation*}
\vertiii{p}_1:=\max_{0\leq k\leq n}\|p\|_{P_k\odot P_{n-k}},\quad \vertiii{p}_2:=\max_{0\leq k\leq n-1}\|p\|_{Z\odot P_k\odot P_{n-k-1}}.
\end{equation*}
Then we prove the following.
\begin{thm}[Theorem \ref{thm: norm 1 and 2}]\label{thm: intro norm 1 and 2}
For any $p\in P_{d,n}$,
\begin{equation*}
\|p\|_{\SArm}\geq\vertiii{p}_2\geq\vertiii{p}_1.
\end{equation*}
\end{thm}
Since our Method 2 recovers the construction of Dixon's lower bound, our lower bound $\vertiii{\cdot}_2$ is at least larger than the one given in \cite{Dix76}. The weak product expressions also match the general heuristic that polynomials with large Schur--Agler norms should be difficult to decompose. Nonetheless, it turns out such lower bounds are not going to show $C(d)=\infty$. In fact, we show that they are equivalent to the Hardy norm, which is smaller than the supremum norm.

\begin{thm}[Theorem \ref{thm:three_norm_two_norm}]
\label{thm:intro three_norm_two_norm}
    Let  $d \in \mathbb N$. There exist constants $C_d \le \binom{2d - 2}{d-1}$
    such that for all $p \in P_{d,n}$, we have
  \begin{equation*}
    \|p\|_2 \le \vertiii{p}_1 \le \sqrt{C_d} \|p\|_2.
  \end{equation*}
  Moreover,
  \begin{equation*}
    \vertiii{p}_1\leq\vertiii{p}_2\leq \sqrt{d}\vertiii{p}_1.
  \end{equation*}
\end{thm}
We list three more methods (Methods 3-5) that may potentially be useful. However, the resulting lower bounds do not have clean expressions. 

Our convex descriptions for the Schur--Agler norms and dual Schur--Agler norms also allow one to numerically approximate these norms using semidefinite programming. In fact, many of the results above are inspired by such numerical experiments. 

This paper is organized as follows. In Section \ref{sec: pre}, we review some basic definitions and results used in this paper. In Section \ref{sec: conv descrip of SA norm}, we explain in more detail the observations mentioned previously and give a convex description of the Schur--Agler norm, a reduced description for homogeneous polynomials, and a parallel description for the supremum norm. In Section \ref{sec: new proofs}, we give new proofs of two known results. Namely, the result of testing von Neumann's inequality on diagonalizable tuples or nilpotent tuples, and the existence of Agler decompositions. In Section \ref{sec: dual SA}, we prove Theorem \ref{thm: intro dual_norm_general} and the corresponding results for homogeneous polynomials and the supremum norm. We also prove some basic properties for the predual Schur--Agler space $(\SA_d)_*$. In Section \ref{sec: methods of constructing L}, we provide 5 methods of constructing $L\in\LLL_c$ using the Hankel operators. The first two methods give the weak product lower bounds for the Schur--Agler norm in Theorem \ref{thm: intro norm 1 and 2}. For Kaijser-Varopoulos-Holbrook type polynomials, i.e., the polynomials
\[
p_t(z)=\sum_{i=1}^d z_i^2+\frac{t}{2}\sum_{i\neq j}z_iz_j,\quad t\in\CC,
\]
we explicitly compute their Schur--Agler norms and dual Schur--Agler norms. In Section \ref{sec: weak product norm}, we prove Theorem \ref{thm:intro three_norm_two_norm}. In Section \ref{sec: SDP}, we explain the numerical methods of approximating $\|\cdot\|_{\SArm}$ and $\|\cdot\|_\ast$ using semidefinite programming. In the Appendix, we compute the $L$ operators for the counterexamples in \cite{CrDa75}\cite{Dix76}\cite{Hol01}\cite{Varo74} and explain how they work.

~

\noindent{\bf Acknowledgments:} The authors would like to thank Catalin Badea, Joseph Ball, Chunlan Jiang, Greg Knese, Dexie Lin, Orr Shalit and Yijun Yao for helpful discussions.

AI tools such as ChatGPT and DeepSeek were used 
to perform literature searches, write code for numerical experiments, 
and perform some preliminary calculations. The final article is entirely 
human generated.

\section{Preliminaries}\label{sec: pre}
In this section, we briefly review some basics about the Hardy space, Hankel operators, weak products, and the Schur--Agler norm.

\subsection{Hardy Space $H^2(\DD^d)$}

For $d\in\NN$, the Hardy space $H^2(\DD^d)$ is the space of holomorphic functions on $\DD^d$ with square summable Taylor coefficients. That is,
\[
H^2(\DD^d):=\left\{f(z)=\sum_{\alpha\in\NN_0^d}a_\alpha z^\alpha~:~\|f\|_2^2=\sum_{\alpha\in\NN_0^d}|a_\alpha|^2<\infty\right\}.
\]
Background on $H^2(\mathbb{D}^d)$ can be found in \cite{RudinPolydisc69}. It can be shown that for any $f\in H^2(\DD^d)$, the radial limit $f^\ast(z):=\lim_{r\to1^-}f(rz)$ converges for almost all $z\in\TT^d$, and the mapping $f\mapsto f^\ast$ gives an isometric embedding into $L^2(\TT^d)$. Therefore it is conventional to identify $f$ with $f^\ast$, and to identify $H^2(\DD^d)$ with a subspace of $L^2(\TT^d)$. Recall that the functions $\{z^\alpha\}_{\alpha\in\ZZ^d}$ form an orthonormal basis of $L^2(\TT^d)$. Then $H^2(\DD^d)$ is the closed linear span of $\{z^\alpha\}_{\alpha\in\NN_0^d}$. It is well known that $H^2(\DD^d)$ is a reproducing kernel Hilbert space on $\DD^d$ with reproducing kernels given by
    \[
    K_z(w)=\prod_{i=1}^d\frac{1}{1-w_i\bz_i}.
    \]
    This means $f(z)=\la f, K_z\ra_2$ for any $f\in H^2(\DD^d)$ and any $z\in\DD^d$. Then $\|K_z\|^2=\prod_{i=1}^d\frac{1}{1-|z_i|^2}$. Denote by $k_z=\frac{K_z}{\|K_z\|}$ the normalized reproducing kernel at $z$. Let $P$ be the Szeg\H{o} projection, i.e., the orthogonal projection from $L^2(\TT^d)$ onto $H^2(\DD^d)$.
    For $\phi\in L^\infty(\TT^d)$, the (small) Hankel operator $\Gamma_\phi$ is defined by
    \[
    \Gamma_\phi: H^2(\DD^d)\to H^2(\DD^d),\quad\Gamma_\phi(f)=P(\phi \cdot Uf),
    \]
    where $Uf(z)=f(\bz)$. Equivalently, $\la\Gamma_\phi f,g\ra_2=\la\phi,\widehat{f}g\ra_2$, where 
    \begin{equation}\label{eqn: f hat}
    \widehat{f}(z)=\overline{f(\bz)}.
    \end{equation}
 The following well-known identity is crucial to our proofs: for any $\phi\in L^\infty(\TT^d)$ and $i=1,\cdots, d$,
 \begin{equation}\label{eqn: Hankel Mz cross commute}
 M_{z_i}^\ast\Gamma_\phi=\Gamma_\phi M_{z_i}.
 \end{equation}

\subsection{Weak products}\label{subsec: weak products}
Recall that $P_{d,n}$ is the space of homogeneous analytic polynomials in $d$ variables of degree $n$. Let $Z_d=\{z_1,\cdots,z_d\}$. For $k, l\in\NN_0$, the following weak product norms will be used in Sections \ref{sec: methods of constructing L} and \ref{sec: weak product norm}.
\[
  \|p\|_{P_{d,k}\odot P_{d,l}}:=\inf\left\{\sum_{j=1}^m\|f_j\|_2\cdot\|g_j\|_2~:~f_j\in P_{d,k},~g_j\in P_{d,l},~\text{and}~p=\sum_{j=1}^m f_jg_j\right\},\quad\forall p\in P_{d,k+l};
\]
and
\[
  \|p\|_{Z_d\odot P_{d,k}\odot P_{d,l}}=\inf\left\{\sum_{i=1}^d\|f_i\|_{P_{d,k}\odot P_{d,l}}~:~f_i\in P_{d,k+l},\text{ and }p=\sum_{i=1}^d z_if_i\right\},\quad\forall p\in P_{d,k+l+1}.
\]
It is elementary to check that
\begin{equation} 
  \label{eqn:WP_easy_estimate}
  \|p\|_{P_{d,k} \odot P_{d,l}} \le \|p\|_{Z_d \odot P_{d,k} \odot P_{d,l-1}};
\end{equation}
this can also be seen from the duality below.
For $n\in\NN$, we define two other norms on $P_{d,n}$ by
\[
    \|q\|_{\Han_k} =\left\|\Gamma_q\big|_{P_{d,k}}\right\|,\quad 0\leq k\leq n,
\]
and
\[
\|q\|_{\Han_k}'=\max_{1\leq i\leq d}\left\|\Gamma_qM_{z_i}\big|_{P_{d,k}}\right\|,\quad 0\leq k\leq n-1.
\]
We will use the following standard duality between Hankel and weak product norm.

\begin{lem}
    \label{lem:Hankel_WP}
    \begin{enumerate}
        \item For $n\in\NN$ and $0\leq k\leq n$, the norms $\|\cdot\|_{\Han_k}$
    and $\|\cdot\|_{P_{d,k} \odot P_{d,n-k}}$ on $P_{d,n}$ are dual to each other with respect to the usual Cauchy pairing. That is,
    the maps
    \[
        (P_{d,n}, \|\cdot\|_{\Han_k}) \to (P_{d,n},\|\cdot\|_{P_{d,k} \odot P_{d,n-k}})^*,
        \quad q \mapsto \langle \cdot, q \rangle_2,
    \]
    and
    \[
        (P_{d,n}, \|\cdot\|_{P_{d,k} \odot P_{d,n-k}}) \to (P_{d,n},\|\cdot\|_{\Han_k})^*,
        \quad q \mapsto \langle \cdot, q \rangle_2,
      \]
      are conjugate linear isometric isomorphisms.
        \item For $n\in\NN$ and $0\leq k\leq n-1$, the norms $\|\cdot\|_{\Han_k}'$
    and $\|\cdot\|_{Z_d\odot P_{d,k} \odot P_{d,n-k-1}}$ on $P_{d,n}$ are dual to each other with respect to the usual Cauchy pairing.
    \end{enumerate}
    
\end{lem}

\begin{proof}
  (1) 
  Since $P_{d,k}$ is finite dimensional, it suffices to show that
  the first map is an isometry. For ease of notation, we write $\langle \cdot,\cdot \rangle_2 = \langle \cdot,\cdot \rangle$.

  Let $q \in P_{d,n}$ and let $h = \sum_j f_j g_j \in P_{d,k} \odot P_{d,n-k}$ with $f_j \in P_{d,k}$ and $g_j \in P_{d,n-k}$.
  Then
  \begin{align*}
    |\langle h, q \rangle| = \Big| \sum_j \langle q, f_j g_j \rangle \Big|
    = \Big| \sum_j \langle \Gamma_q \widehat{f}_j,  g_j \rangle \Big|
    &\le \sum_j \|\Gamma_q \widehat{f}_j\|_2 \|g_j\|_2 \\
    &\le \|q\|_{\Han_k} \sum_j \|f_j\|_2 \|g_j\|_2.
  \end{align*}
  Taking the infimum over all such representations of $h$, we obtain
  \begin{equation*}
    |\langle h, q \rangle| \le \|q\|_{\Han_k} \|h\|_{P_{d,k} \odot P_{d,n-k}}.
  \end{equation*}
  Conversely, since $\Gamma_q$ maps $P_{d,k}$ into ${P_{d,n-k}}$, we have
  \begin{align*}
    \|q\|_{\Han_k} &= \sup \{ |\langle \Gamma_q \widehat{f}, g \rangle|: f \in P_{d,k}, g \in P_{d,n-k}, \|f\|_2, \|g\|_2 \le 1\} \\
                   &= \sup \{ |\langle q, fg \rangle|: f \in P_{d,k}, g \in P_{d,n-k}, \|f\|_2, \|g\|_2 \le 1\} \\
                   &\le \sup \{ |\langle h, q \rangle|: h \in P_{d,k} \odot P_{d,n-k}, \|h\|_{P_{d,k} \odot P_{d,n-k}} \le 1\}.
  \end{align*}
  This shows that
  \begin{equation*}
    \|q\|_{\Han_k} = \sup \{ |\langle h, q \rangle|: h \in P_{d,k} \odot P_{d,n-k}, \|h\|_{P_{d,k} \odot P_{d,n-k}} \le 1\}.
  \end{equation*}

  (2) Let $q \in P_{d,n}$ and let $h = \sum_i z_i f_i \in P_{d,n}$. Using the fact that $\Gamma_q M_{z_i} = \Gamma_{M_{z_i}^* q}$
  and (1), we find that
  \begin{align*}
    |\langle h,q \rangle| \le \sum_i | \langle z_i f_i,q \rangle|
    &= \sum_i | \langle f_i, M_{z_i}^* q \rangle|
    \le \sum_i \|f_i\|_{P_{d,k} \odot P_{d,n-k-1}} \|\Gamma_{M_{z_i}^* q} \big|_{P_{d,k}} \| \\
    &\le \|q\|_{\Han_k}' \sum_i \|f_i\|_{P_{d,k} \odot P_{d,n-k-1}}.
  \end{align*}
  Hence
  \begin{equation*}
    | \langle h,q \rangle| \le \|q\|_{\Han_k}' \|h\|_{Z_d \odot P_{d,k} \odot P_{d,n-k-1}}.
  \end{equation*}
  On the other hand,
  \begin{align*}
    \Big\| \Gamma_q M_{z_i} \Big|_{P_{d,k}} \Big\|
    &= \sup \{ | \langle f, M_{z_i}^* q \rangle|: f \in P_{d,k} \odot P_{d,n-k-1}, \|f\|_{P_{d,k} \odot P_{d,n-k-1}} \le 1 \} \\
    &= \sup \{ | \langle z_i f, q \rangle|: f \in P_{d,k} \odot P_{d,n-k-1}, \|f\|_{P_{d,k} \odot P_{d,n-k-1}} \le 1 \} \\
    &\le \sup \{ | \langle h, q \rangle|: h \in Z_d \odot  P_{d,k} \odot P_{d,n-k-1}, \|h\|_{Z_d \odot P_{d,k} \odot P_{d,n-k-1}} \le 1 \}.
  \end{align*}
  This shows that
  \begin{equation*}
    \|q\|_{\Han_k}' = \sup \{ | \langle h,q \rangle| : h \in Z_d \odot P_{d,k} \odot P_{d,n-k-1}, \|h\|_{Z_d \odot P_{d,k} \odot P_{d,n-k-1}} \le 1 \}. \qedhere
  \end{equation*}
\end{proof}

\begin{rem}
  It is known that weak product spaces can be regarded
  as quotients of the space of trace class operators; see for instance \cite[Section 2]{AHM+18}.
  In our case, this can be made very explicit, and this can be used
  to compute upper bounds of the weak product norm. Explicitly, let $p \in P_{d,n}$, say $p(z) =  \sum_{|\gamma| = n} c_\gamma z^\gamma$, and let $0 \le k \le n$.
  Let us order the monomials $z^\alpha$ of degree $k$ in some way and write $[z^\alpha]_{|\alpha| = k}$ for the resulting column
  vector. Then we can represent
  \begin{equation*}
    p(z) =
    \begin{bmatrix}
      z^\alpha
    \end{bmatrix}_{|\alpha|=k}^T A
    \begin{bmatrix}
      z^\beta
    \end{bmatrix}_{|\beta|=n-k},
  \end{equation*}
  where $A = [a_{\alpha,\beta}]$ is a scalar matrix of the appropriate size. Note that a scalar
  matrix $A$ represents $p$ in this way if and only if $c_\gamma = \sum_{\alpha + \beta = \gamma} a_{\alpha \beta}$
  for all $\gamma$.

  If $A = \sum_{j} v_j  w_j^T$ is a decomposition of $A$ into rank one matrices,
  then we obtain a weak factorization $p = \sum_{j} f_j g_j$, where
  \begin{equation*}
    f_j =
    \begin{bmatrix}
      z^\alpha
    \end{bmatrix}^T v_j
    \quad \text{ and } \quad
    g_j = w_j^T
    \begin{bmatrix}
      z^\beta
    \end{bmatrix}.
  \end{equation*}
  Since $\|f_j\|_2 = \|v_j\|$ and $\|g_j\|_2 = \|w_j\|$, this shows that
  \begin{equation*}
    \|p\|_{P_k \odot P_{n-k}} \le \|A\|_{\mathrm{nuc}}.
  \end{equation*}
  Conversely, any weak factorization of $p$ yields a nuclear decomposition of some representing matrix $A$ of $p$,
  so $\|p\|_{P_k \odot P_{n-k}}$ is the infimum of the nuclear norms of representing matrices of $p$.
  This also shows that in the definition of the weak product norm on $P_{k} \odot P_{n-k}$,
  it suffices to consider weak factorizations of length $m \le \min( \dim(P_k),\dim(P_{n-k}))$.
\end{rem}

\subsection{Schur--Agler space}
For a polynomial $p\in\polyn$, its Schur--Agler norm is given by
\[
\|p\|_{\SArm}:=\sup\left\{\|p(\Tbf)\|~:~\Tbf\text{ is a commuting }d\text{ tuple of contractions}\right\}.
\]
More generally, we define the Schur--Agler norm of a holomorphic function $f: \mathbb{D}^d \to \mathbb{C}$
by
\begin{equation*}
  \|f\|_{\SArm} = \sup \{ \|f(\Tbf)\|: \Tbf \text{ is a commuting $d$-tuple of strict contractions} \},
\end{equation*}
where $f(\Tbf)$ is defined (for instance) by plugging $\Tbf$ into the power series expansion of $f$ at $0$.
The Schur--Agler space $\SA_d$ consists of all holomorphic $f$ for which $\|f\|_{\SArm} < \infty$.
Clearly, $\SA_d \subset H^\infty(\mathbb{D}^d)$.

\begin{lem}
  \label{lem:SA_approx}
  Let $f \in \SA_d$ and let $p_n$ be the $n$-th Fej\'er mean of $f$. Then $\|p_n\|_{\SArm} \le \|f\|_{\SArm}$,
  for all $n \in \mathbb{N}$ and $p_n \to f$ in $H^2(\mathbb{D}^d)$ as $n \to \infty$.
\end{lem}

\begin{proof}
  This follows from standard properties of the Fej\'er kernel; see for instance
  \cite[Section 1.2]{Katznelson04}. Indeed,
  let $F_n(z) = \sum_{k=-n}^n (1 - \frac{|k|}{n+1}) z^k$ be the Fej\'er kernel in one variable. Thus,
  if $f \in \SA_d$ with homogeneous decomposition $f = \sum_{k=0}^\infty f_k$, then
  \[
    p_n(z) = \sum_{k=0}^n \Big( 1 - \frac{k}{n+1} \Big) f_k (z) = \int_{\mathbb T} f(z \bar \lambda) F_n(\lambda) d \sigma(\lambda). 
  \]
  The first representation of $p_n$ easily implies that $p_n \to f$ in $H^2(\mathbb D^d)$ as $n \to \infty$. The integral in the second representation converges uniformly on compact subsets of $\mathbb D^d$ in $z$, so if $\Tbf$ is a tuple of strict commuting contractions, then
  \[
    p_n(T) = \int_{\mathbb T} f( \bar \lambda \Tbf) F_n(\lambda) d \sigma(\lambda),
  \]
and the integral converges in operator norm.
Since $\|f(\bar \lambda \Tbf)\| \le \|f\|_{\SArm}$ for all $\lambda \in \mathbb T$,
we can use positivity of the Fej\'er kernel and the triangle inequality
to conclude that
\[
    \|p_n(T)\| \le \int_{\mathbb T} \|f\|_{\SArm} F_n(\lambda) d \sigma(\lambda) = \|f\|_{\SArm}. \qedhere
\]
\end{proof}

\section{A Convex Description of the Schur--Agler Norm}\label{sec: conv descrip of SA norm}

In this section, we give a description of the Schur--Agler norm using a cone of positive operators on $H^2(\DD^d)$. For homogeneous polynomials, we reduce the expression to a smaller cone. We also give a parallel description for the supremum norm. The results in this section form the foundation for subsequent sections.
\begin{construction}\label{defn: cyclic commuting tuple}
Let $\Tbf=(T_1,\cdots,T_d)$ be a commuting $d$-tuple of operators on a Hilbert space $\HHH$. Recall that a vector $\xi\in\HHH$ is called a cyclic vector for $\Tbf$ if $\{p(\Tbf)\xi~:~p\in\polyn\}$ is dense in $\HHH$. Let us call $(\Tbf,\xi)$ a cyclic commuting ($d$-)tuple. We say that two cyclic commuting tuples $(\Tbf,\xi)$ and $(\Tbf',\xi')$ are unitarily equivalent if there is a unitary operator $U$ such that $UT_i=T'_iU$ and $U\xi=\xi'$. 
For a cyclic commuting tuple $(\Tbf,\xi)$, define the semi-inner product $\la\cdot,\cdot\ra_{\Tbf,\xi}$ on $\polyn$ by
\[
\la p, q\ra_{\Tbf,\xi}=\la p(\Tbf)\xi, q(\Tbf)\xi\ra,\quad\forall p, q\in\polyn.
\]
Let $\III=\{p\in\polyn~:~\la p, p\ra_{\Tbf,\xi}=0\}=\{p\in\polyn~:~p(\Tbf)=0\}$, equip $\polyn/\III$ with the inner product induced by $\la\cdot,\cdot\ra_{\Tbf,\xi}$, and denote by $\HHH'$ its completion. Define
\[
T_i':\HHH'\to\HHH',\quad\text{determined by}\quad p+\III\mapsto z_ip+\III,\quad\forall p\in\polyn,~\forall i=1,\cdots,d.
\]
Take $\xi'=1+\III$. Then it is elementary to check that $(\Tbf,\xi)$ is unitarily equivalent to $(\Tbf',\xi')$. In other words, one can reconstruct the tuple $(\Tbf,\xi)$ using the semi-inner product $\la\cdot,\cdot\ra_{\Tbf,\xi}$. 
\end{construction}

In this paper, we take one step further in the construction above. The idea is to express the semi-inner product $\la\cdot,\cdot\ra_{\Tbf,\xi}$ as a positive operator on an analytic function space containing $\polyn$, and to take advantage of the operator and function theoretic tools on this space. 
\begin{construction}\label{defn: L operator}
Suppose $(\Tbf,\xi)$ is a cyclic commuting $d$-tuple and assume that there exists a constant $C \ge 0$
such that $\|p(T) \xi\| \le C \|p\|_2$ for all $p \in \mathbb{C}[\mathbf{z}]$.
Then the following equation determines a bounded linear operator $L_{\Tbf,\xi}$ on the Hardy space $H^2(\DD^d)$:
\[
\la L_{\Tbf,\xi}p, q\ra_2=\la p(\Tbf)\xi, q(\Tbf)\xi\ra_\HHH,\quad\forall p, q\in \polyn.
\]
For example, this holds when each $T_i$ is a strict contraction, in which case we actually have
\[
\la L_{\Tbf,\xi} f,g\ra_2=\la f(\Tbf)\xi, g(\Tbf)\xi\ra_\HHH,\quad\forall f, g\in H^2(\DD^d).
\]
Let
\[
\LLL_c=\left\{L\in\BBB(H^2(\DD^d))~:~L\geq0,~M_{z_i}^\ast LM_{z_i}\leq L,~i=1,\cdots,d\right\}.
\]
If $(\Tbf,\xi)$ is a cyclic commuting tuple of contractions (and if $L_{\Tbf,\xi}$ is bounded), then it is elementary to check that $L_{\Tbf,\xi}\in\LLL_c$.

Conversely, given $L\in\LLL_c,~L\neq0$, define a semi-inner product by 
\[
\la p, q\ra_L=\la Lp, q\ra,\quad\forall p, q\in\polyn.
\]
Then Construction \ref{defn: cyclic commuting tuple} yields a cyclic commuting tuple of contractions $(\Tbf,\xi)$ with $L=L_{\Tbf,\xi}$, so
\[
    \la L p,q \ra_2 = \la p(\Tbf) \xi, q(\Tbf) \xi \ra, \quad \forall p,q \in \polyn.
\]
Alternatively, write $L=A^*A$ and let
\[
T_i: \overline{\mathrm{Range}A}\to\overline{\mathrm{Range}A},\quad\text{determined by}\quad A(f)\mapsto A(z_if),\quad\forall f\in H^2(\DD^d).
\]
Then it is easy to verify that $\Tbf$ is a commuting tuple of contractions with cyclic vector $\xi=A1$, and $L=L_{\Tbf,\xi}$. Up to unitary equivalence, the tuple $(\Tbf,\xi)$ is unique. Let us call it the cyclic commuting tuple associated with $L$.
\end{construction}

The following key lemma will be used throughout this paper.
\begin{lem}\label{lem: SA norm with Lc}
For any $f\in H^2(\mathbb{D}^d)$,
\[
\|f\|_{\SArm}^2=\sup\left\{\frac{\la L f, f\ra_2}{\la L1,1\ra_2}~:~L\in\LLL_c, L \neq 0  \right\},
\]
where both sides of the equality are allowed to be infinite.
\end{lem}
\begin{proof}
First, note that if $L \in \mathcal{L}_c$ with $\langle L 1,1 \rangle_2 = 0$, then the commuting tuple associated with $L$ is zero, hence $L = 0$. So the right hand side is well-defined. Let $f \in H^2(\DD^d)$. Let $L \in \mathcal{L}_c$. For $0<r<1$, define
\[
D_r: H^2(\DD^d)\to H^2(\DD^d),\quad D_r(h)=h_r,\quad\text{where }h_r(z)=h(rz).
\]
Write $L_r=D_rLD_r$. Then  
\[
L_r\in\LLL_c,\quad M_{z_i}^\ast L_r M_{z_i} \leq r^2 L_r,
\]
and $L_r\to L$ in the strong operator topology as $r\to 1^-$.
Let $(\Tbf,\xi)$ be the cyclic commuting tuple associated with $L_r$. Then $\|T_i\|\leq r$, and
\[
\|f\|_{\SArm}^2\geq\|f(\Tbf)\|^2\geq\frac{\|f(\Tbf)\xi\|^2}{\|\xi\|^2}=\frac{\la L_rf,f\ra}{\la L_r1,1\ra}\to\frac{\la Lf,f\ra}{\la L1,1\ra},\quad r\to1^-.
\]
This proves the inequality ``$\geq$'' in the statement.

 Conversely, note that
  \begin{equation*}
    \|f\|_{\SArm} = \sup \Big\{ \frac{\|f(\Tbf)\xi\|}{\|\xi\|}: (\Tbf,\xi) \text{ cyclic commuting $d$-tuple of strict contractions} \Big\}.
  \end{equation*}
  Let $(\Tbf,\xi)$ be a cyclic commuting $d$-tuple of strict contractions and let $L_{\Tbf,\xi} \in \mathcal{L}_c$ be the associated
  operator. Then
  \begin{equation*}
    \frac{\|f(T) \xi\|^2}{\|\xi\|^2} = \frac{\langle Lf,f \rangle_2 }{\langle L1,1 \rangle_2}
  \end{equation*}
 This proves the inequality ``$\le$'' in the statement.
\end{proof}

Let $P_n = P_{d,n}$ be the space of homogeneous polynomials of degree $n$ in $d$ variables.
For homogeneous polynomials, Lemma \ref{lem: SA norm with Lc} can be improved in the sense that the cone $\mathcal{L}_c$
can be replaced by a finite dimensional cone.
We define
\begin{equation*}
  \mathcal{L}_c^{(n)} = \left\{ L \in \mathcal{L}_c: L = \sum_{k=0}^n L_k, \text{ where } L_k \in \mathcal{B}(P_k) \right\}.
\end{equation*}
In other words, $\LLL_c^{(n)}$ consists of those elements of $\LLL_c$ that reducing each $P_k$ and vanishing on $P_{\le n}^\perp$.
Note that an operator $L = \sum_{k=0}^n L_k$, where each $L_k \in \mathcal{B}(P_k)$ is positive,
belongs to $\LLL_c$ if and only if
\begin{equation*}
  M_{z_i}^* L_k M_{z_i} \le L_{k-1} \quad \text{ for all } k=1,\ldots,n, \quad i=1,\ldots,d.
\end{equation*}

\begin{lem}\label{lem: SA by graded}
For any $p\in P_n$,
\[
  \|p\|_{\SArm}^2=\sup\left\{\frac{\la Lp,p\ra_2}{\la L1,1\ra_2}~:~L\in\LLL_c^{(n)},~L \neq 0\right\}.
\]
\end{lem}

\begin{proof}
We abuse notation and write $P_k$ for the space of homogeneous polynomials
of degree $k$ and for the orthogonal projection onto it.
Let $L \in \LLL_c$.
In light of Lemma \ref{lem: SA norm with Lc}, it suffices to show
that there exists $L' \in \LLL_c^{(n)}$ with 
$\la L p, p \ra_2 = \la L' p, p \ra_2$ and $\la L 1,1 \ra_2 = \la L' 1,1 \ra_2$.
Let $L' = \sum_{k=0}^n P_k L P_k$. To show that $L' \in \LLL_c$, we have
to check that $M_{z_i}^* L' M_{z_i} \le L'$ for each $i$.

Since $P_k M_{z_i} = M_{z_i} P_{k-1}$ for $k=1,\ldots,n$ and $P_0 M_{z_i} = 0$, we have
\begin{align*}
    M_{z_i}^* L' M_{z_i} = \sum_{k=0}^n M_{z_i}^* P_k L P_k M_{z_i}
    = \sum_{k=1}^n P_{k-1} M_{z_i}^* L M_{z_i} P_{k-1}
    \le \sum_{k=1}^n P_{k-1} L P_{k-1} \le L'
\end{align*}
as desired.
\end{proof}

For comparison, we establish a related expression for the supremum norm. Let
\begin{equation*}
  \mathcal{L}_t = \{ L \in B(H^2(\mathbb{D}^d)): L \ge 0, M_{z_i}^* L M_{z_i} = L, i=1,\cdots,d \}.
\end{equation*}
Clearly, $\mathcal{L}_t \subset \mathcal{L}_c$. By the analogue of the Brown--Halmos criterion on the polydisc \cite[Theorem 3.1]{MSS18},
$\mathcal{L}_t$ precisely consists of all positive Toeplitz operators on $H^2(\mathbb{D}^d)$.

\begin{lem}
  \label{lem:sup_norm_Lt}
  For any $f \in H^2(\mathbb{D}^d)$,
  \[    
  \|f\|_{\infty}^2=\sup\left\{\frac{\la L f, f\ra_2}{\la L1,1\ra_2}~:~L\in\LLL_t, L \neq 0  \right\},
  \]
  where both sides of the equality are allowed to be infinite.
\end{lem}

\begin{proof}
  As mentioned above, $\mathcal{L}_t$ precisely consists of all positive Toeplitz operators on $H^2(\mathbb{D}^d)$. Moreover,
  if $T_w \in \mathcal{L}_t$, then
  \begin{equation*}
    \langle T_w f,f \rangle_2 = \int_{\mathbb{T}^d} |f|^2 w \, d \sigma.
  \end{equation*}
  Thus,
  \begin{equation*}
  \sup\left\{\frac{\la L f, f\ra_2}{\la L1,1\ra_2}~:~L\in\LLL_t, L \neq 0  \right\}
  = \sup \left\{ \frac{\int_{\mathbb{T}^d} |f|^2 w \, d \sigma}{\int_{\mathbb{T}^d} w  \, d \sigma}: w \in L^\infty(\mathbb{T}^d), w \ge 0, w \neq 0 \right\}.
  \end{equation*}
  The supremum on the right is easily seen to equal the essential supremum of $|f|^2$ on $\mathbb{T}^d$, which equals $\|f\|_\infty^2$.
\end{proof}

\begin{rem}
  If $d \le 2$, then the supremum norm and the Schur--Agler norm  agree, hence so do the expressions
  in Lemma \ref{lem: SA norm with Lc} and Lemma \ref{lem:sup_norm_Lt}.
  We now explain how the dilation theorem of Sz.-Nagy is reflected in our picture, thus relating
  $\mathcal{L}_t$ and $\mathcal{L}_c$ more directly in case $d=1$ without using the statements of the two lemmas.

  Since $\mathcal{L}_t \subset \mathcal{L}_c$, it is clear that for every $p \in \mathbb C[z]$,
  \begin{equation}
    \label{eqn:dilation_0}
  \sup\left\{\frac{\la L p, p\ra_2}{\la L1,1\ra_2}~:~L\in\LLL_t, L \neq 0  \right\}
  \le
  \sup\left\{\frac{\la L p, p\ra_2}{\la L1,1\ra_2}~:~L\in\LLL_c, L \neq 0  \right\}.
  \end{equation}

  Conversely, let $L \in \mathcal{L}_c$ and assume that the associated cyclic operator $(T,h)$ is a strict contraction.
  (It suffices to take the supremum over such $L \in \mathcal{L}_c$.)
  Let $U$ be the minimal unitary dilation of $T$ and let $\mu$ be spectral measure of $U$ associated with the vector $h$, i.e.\
  \begin{equation*}
    \langle q(U)^* p(U) h,h \rangle = \int_{\mathbb{T}} \overline{q} p \, d \mu \quad \text{ for all } p,q \in \mathbb{C}[z].
  \end{equation*}
  Since $T$ is a strict contraction, $\mu$ is absolutely continuous (see \cite[Theorem II.1.2]{SFB+10}), say $d \mu = w d \sigma$,
  where $w \in L^1(\mathbb{T})$ is non-negative. Assume, for the moment, that $w \in L^\infty(\mathbb{T})$
  and let $\widehat{L} = T_w$ be the Toeplitz operator with symbol $w$. Then $\widehat{L} \in \mathcal{L}_t$,
  \begin{equation}
    \label{eqn:dilation}
    \begin{split}
      \langle \widehat{L} 1,1 \rangle_2 &= \int_{\mathbb{T}} w \, d \sigma = \int_{\mathbb{T}} 1 \, d \mu = \|h\|^2 = \langle L1,1 \rangle_2 \quad \text{ and } \\
    \langle \widehat{L} p, p \rangle_2 &= \int_{\mathbb{T}} |p|^2 w \, d \sigma = \|p(U) h \|^2 \ge \|p(T) h \|^2
    = \langle L p,p \rangle_2
    \end{split}
  \end{equation}
  for all $p \in \mathbb{C}[z]$, and so $L \le \widehat{L}$.
  Thus, we obtain the reverse inequality of \eqref{eqn:dilation_0} whenever $w \in L^\infty(\mathbb{T})$. If $w$ merely belongs to $L^1(\mathbb{T})$,
  we let $\widehat{L}_n$ be the Toeplitz operator with symbol $\min(w,n)$ for $n \in \mathbb{N}$.
  Then \eqref{eqn:dilation} is true in the limit $n \to \infty$, which establishes the reverse inequality in general.
\end{rem}

A similar argument as in the proof of Lemma \ref{lem: SA by graded} shows that
when computing the supremum norm of a homogeneous polynomial with the help of Lemma \ref{lem:sup_norm_Lt}, it suffices
to consider graded operators $L = \sum_{k=0}^\infty L_k \in \mathcal{L}_t$,
where each $L_k \in \mathcal{B}(P_k)$.
However, we can no longer consider finite sums $L = \sum_{k=0}^n L_k$,
where each $L_k \in \mathcal{B}(P_k)$ is positive, satisfying
\begin{equation}
  \label{eqn:TTO}
  M_{z_i}^* L_k M_{z_i} = L_{k-1} \quad \text{ for all } k=1,\ldots,n, \quad i=1,\ldots,d.
\end{equation}
The issue is that such operators $L$ may fail to extend to positive Toeplitz operators.
The following is a concrete example where the analogue of Lemma \ref{lem: SA by graded} for the supremum norm fails.
The example was found through numerical experiments using semi-definite programming (see also Section \ref{sec: SDP}), but it can be verified symbolically,
and in principle even by hand.

\begin{exam}
  \label{exa:TTO}
  Let
  \begin{equation*}
    p(z_1,z_2,z_3) = z_1^2 + z_2^2 + z_3^2 - 2 (z_1 z_2 + z_1 z_3 + z_2 z_3)
  \end{equation*}
  be the Kaijser--Varopoulos polynomial. It is known that $\|p\|_\infty = 5$; see for instance \cite[Proposition 2]{Hol01}
  for a proof.

  We equip $P_0$ with the basis $\{1\}$ and $P_1$ with the basis $\{z_1,z_2,z_3\}$ and define operators
  \begin{equation*}
    \renewcommand{\arraystretch}{1.3}
    L_0  = 1 \quad \text{ and } \quad
    L_1  =
    \begin{bmatrix}
      1 & -\frac{1}{5} & -\frac{1}{5} \\
      -\frac{1}{5} & 1 & -\frac{1}{5} \\
      -\frac{1}{5} & -\frac{1}{5} & 1
    \end{bmatrix}.
  \end{equation*}
  Moreover, equip $P_2$ with the basis $\{z_1^2,z_2^2,z_3^2,z_1 z_2,z_1 z_3,z_2 z_3 \}$
  and define
  \begin{equation*}
    \renewcommand{\arraystretch}{1.3}
    L_2 = \begin{bmatrix}
1 & \frac{7}{10} & \frac{7}{10} & -\frac{1}{5} & -\frac{1}{5} & -\frac{4}{5} 
\\
 \frac{7}{10} & 1 & \frac{7}{10} & -\frac{1}{5} & -\frac{4}{5} & -\frac{1}{5} 
\\
 \frac{7}{10} & \frac{7}{10} & 1 & -\frac{4}{5} & -\frac{1}{5} & -\frac{1}{5} 
\\
 -\frac{1}{5} & -\frac{1}{5} & -\frac{4}{5} & 1 & -\frac{1}{5} & -\frac{1}{5} 
\\
 -\frac{1}{5} & -\frac{4}{5} & -\frac{1}{5} & -\frac{1}{5} & 1 & -\frac{1}{5} 
\\
 -\frac{4}{5} & -\frac{1}{5} & -\frac{1}{5} & -\frac{1}{5} & -\frac{1}{5} & 1 
    \end{bmatrix}.
  \end{equation*}
  One checks that these operators satisfy \eqref{eqn:TTO} (e.g.\ $M_{z_1}^* L_2 M_{z_1}$ is the submatrix of $L_2$ corresponding
  to rows and columns with indices $1,4,5$, which equals $L_1$).
  Moreover, $L_0,L_1,L_2$ are positive ($L_2$ has eigenvalues $0,0,0,\frac{3}{2},\frac{3}{2},3$).
  We have $\langle L1,1 \rangle_2  =1$. But $\langle L p, p \rangle_2 = \frac{144}{5} = 28.8 > 25 = \|p\|^2_\infty$.
\end{exam}

\begin{rem}\label{rem: two pairings}
The approach of the Schur--Agler norm using convex arguments is not new. In \cite{Agler82} (also see \cite[Section 2.8]{AgMcYo20book}), Agler defined the space $\mathrm{Her}(\DD^d)$ of hereditary functions and realized a contractive cyclic commuting tuple $(\Tbf,\xi)$ as a bounded linear functional $\Lambda_{\Tbf,\xi}\in\left(\mathrm{Her}(\DD^d)\right)^*$. The linear functional $\Lambda_{\Tbf,\xi}$ is determined by
\[
\Lambda_{\Tbf,\xi}(z^\alpha\bw^\beta)=\la\Tbf^\alpha\xi, \Tbf^\beta\xi\ra,\quad\forall\alpha, \beta\in\NN_0^d.
\]
By linear extension, this determines the value of $\Lambda_{\Tbf,\xi}$ on $\CC[\zbf,\overline{\mathbf{w}}]$, which is dense in $\mathrm{Her}(\DD^d)$. Essentially, the cone of the linear functionals $\{\Lambda_{\Tbf,\xi}~:~(\Tbf,\xi)\text{ is a contractive commuting cyclic tuple}\}$ plays a similar role as $\LLL_c$ in our case. In Subsection \ref{subsec: Agler decomp}, we give a proof of the existence of Agler decomposition using the cone $\LLL_c$.
\end{rem}

\section{First Applications: New Proofs}\label{sec: new proofs}
In this section, we give two examples of how the extra structure coming from the $L$ operators prove convenient when studying contractive commuting operator tuples. We will give two new proofs of existing results. 
\subsection{Testing the Schur--Agler Norm on Diagonal or Nilpotent Tuples}
The goal of this subsection is to give a new proof of Proposition \ref{prop:nilpotent-diagonal-intro}.
Let $\Tbf$ be a tuple of operators similar to $M_z$ on $H^2(\mathbb{D}^d)$, i.e.\ there exists
an invertible operator $S$
such that $T_j = S^{-1} M_{z_j} S$ for $j=1,\ldots,d$. If $f \in H^\infty(\mathbb{D}^d)$, then we define
$f(T) = S^{-1} f(M_z) S = S^{-1} M_f S$.
\begin{prop}
  \label{prop:similar_to_M_z}
  For $f \in H^\infty(\mathbb{D}^d)$,
  \begin{equation*}
    \|f\|_{\SArm} = \sup \{ \|f(\Tbf)\|: \Tbf \text{ is a tuple of contractions similar to $M_z$ on $H^2(\mathbb{D}^d)$}\}.
  \end{equation*}
\end{prop}

\begin{proof}
  The inequality ``$\ge$'' is obvious when $f \in \polyn$,
  and an approximation argument (for instance using Lemma \ref{lem:SA_approx}) shows that it holds for all $f \in H^\infty(\mathbb D^d)$.
  
  Conversely,
  by Lemma \ref{lem: SA norm with Lc}, we have
  \begin{equation*}
    \|f\|_{\SArm}^2 = \sup \Big\{ \frac{\langle Lf,f \rangle_2}{\langle L1,1 \rangle_2}: L \in \mathcal{L}_c \Big\}.
  \end{equation*}
  Let $L \in \mathcal{L}_c$, $\varepsilon > 0$ and set $L_\varepsilon = L + \varepsilon I$.
  Then $L_\varepsilon \in \mathcal{L}_c$,
  $L_\varepsilon$ is invertible,
  and clearly $L_\varepsilon \to L$ in norm as $\varepsilon \to 0$.
  Thus,
  \begin{equation*}
    \|f\|_{\SArm}^2 \le \sup \Big\{ \frac{\langle Lf,f \rangle_2}{\langle L 1,1 \rangle_2} : L \in \mathcal{L}_c, L \text{ invertible} \Big\}.
  \end{equation*}
  Now, let $L \in \mathcal{L}_c$ be invertible
  and let $(\Tbf,\xi)$ be the cyclic commuting tuple
  of contractions induced by $L$. By Construction \ref{defn: L operator}, $T_i=L^{1/2}M_{z_i}L^{-1/2}, i=1, \cdots, d$;
  thus $\Tbf$ is similar to $M_z$.
  Moreover, for all $f \in \polyn$,
  \begin{equation*}
    \frac{\langle L f,f \rangle_2}{\langle L1,1 \rangle_2} = \frac{\|f(\Tbf) \xi\|^2}{\|\xi\|^2} \le \|f(\Tbf)\|^2,
  \end{equation*}
  and again an approximation argument
  shows that the equality on the left, and hence the inequality, continues to hold for all $f \in H^\infty(\mathbb{D}^d)$.
  This proves the remaining inequality.
\end{proof}

Recall that we write $\mathcal{C}_d$ for the set of cyclic commuting $d$-tuples of strictly contractive matrices (of arbitrary size).

\begin{prop}\label{prop: nilpotent and diagonal}
For $f\in H^\infty(\mathbb{D}^d)$,
\begin{flalign*}
    \|f\|_{\SArm}=&\sup\left\{\|f(\Tbf)\|~:~\Tbf \in \mathcal{C}_d \text{ is jointly nilpotent}\right\}\\
    =&\sup\left\{\|f(\Tbf)\|~:~\Tbf \in \mathcal{C}_d \text{ is jointly diagonalizable}\right\}.
\end{flalign*}
\end{prop}

\begin{proof}
  Let $\Tbf$ be a tuple of contractions similar to $M_z$ on $H^2(\mathbb{D}^d)$, i.e.\ $\Tbf = S^{-1} M_z S$.
  For each finite set $F \subset \mathbb{D}^d$, let $Q_F = \mathrm{span} \{K_{z}: z \in F \}$.
  Then $Q_F$ is invariant under the tuple $M_z^*$, hence $M_F := S^* Q_F$ is invariant under $\Tbf^*$.
  Moreover, $M_z^* \big|_{Q_F}$ is a diagonalizable tuple on a finite dimensional space, hence so is $\Tbf^* \big|_{M_F}$
  and thus also $P_{M_F} \Tbf \big|_{M_F}$. Note also that $P_{M_F} \Tbf \big|_{M_F}$ is cyclic as $\Tbf$ is cyclic
  and $M_F$ is co-invariant.
  The net $(P_{M_F})$ tends to the identity in SOT, so
  \begin{equation*}
    \|f(\Tbf)\| \le \sup_F \|P_{M_F} f(\Tbf) \big|_{M_F}\| = \sup_F \| f( P_{M_F} \Tbf \big|_{M_F})\|,
  \end{equation*}
  the last inequality by co-invariance of $M_F$.
  In light of Proposition \ref{prop:similar_to_M_z}, it follows that
  \begin{equation*}
    \|f\|_{\SArm}
    \le \sup\left\{\|f(\Tbf)\|~:~\Tbf\text{ is a commuting tuple of cyclic diagonalizable contractive matrices}\right\};
  \end{equation*}
  by approximating $\Tbf$ with $r \Tbf$ for $r < 1$, we see that
  \begin{equation*}
    \|f\|_{\SArm} \le \sup \{ \|f(\Tbf)\|: \Tbf \in \mathcal{C}_d \text{ is diagonalizable} \}.
  \end{equation*}
  The reverse inequality is trivial.

  The proof in the nilpotent case is similar; replace $Q_F$ with
  \begin{equation*}
    Q_n = \mathrm{span} \{ z^\alpha: |\alpha| \le n \}. \qedhere
  \end{equation*}
\end{proof}

\begin{rem}
  The $L$-operator framework and Lemma \ref{lem: SA norm with Lc} make it possible to give a uniform proof
  of both expressions in Proposition \ref{prop: nilpotent and diagonal}. Working with the operators directly,
  the reduction to nilpotent matrices is arguably easier than the reduction to diagonalizable matrices.
  For instance, one can tensor a commuting tuple $\mathbf T$ with the unilateral shift $M_z \in B( H^2(\mathbb{D}))$
  and then approximate $M_z$ by $P_{1,\le n} M_z \big|_{P_{1, \le n}}$,
  which produces nilpotent tuples. (This is the proof of the first author mentioned in \cite{knese25}.)
  But such a procedure will not produce diagonalizable tuples in general, for instance if $\mathbf{T}$ is nilpotent to begin with,
  since the tensor product remains nilpotent.
  In the $L$-operator framework, we have the very simple approximation of $L$ by $L + \varepsilon I$, which avoids this problem.
\end{rem}

\subsection{The Agler Decomposition}\label{subsec: Agler decomp}
In \cite{Agler82}, Agler proved the following result. The decomposition on the right hand side is now widely known as the Agler decomposition. 
\begin{thm}[Agler \cite{Agler82}]\label{thm: Agler decomp}
A bounded holomorphic function $f$ on $\DD^d$ satisfies $\|f\|_{\SArm}\leq1$ if and only if there exist positive definite kernels $k_1, \cdots, k_d$ such that
\[
1-f(z)\overline{f(w)}=\sum_{i=1}^d(1-z_i\bw_i)k_i(z,w),\quad\forall z, w\in\DD^d.
\]
\end{thm}

For a hereditary function $h(z,w)=\sum_{j=1}^mf_j(z)\overline{g_j(w)}$, where $f_j, g_j\in H^2(\DD^d)$, define the operator
\[
R_h=\sum_jf_j\otimes g_j\quad\text{on}\quad H^2(\DD^d).
\]
Note that $R_h$ does not depend on the decomposition of $h$, since
\begin{equation}\label{eqn: Rh and h via kernels}
\la R_h K_w, K_z\ra=h(z,w),\quad\forall z, w\in\DD^d.
\end{equation}
We give a proof of the ``only if'' part of Theorem \ref{thm: Agler decomp}. (The ``only if'' direction is generally considered the more difficult direction.)

\begin{proof}[Proof of Theorem \ref{thm: Agler decomp}, ``$\Rightarrow$'']
Suppose $f\in H^\infty(\DD^d)$ and $\|f\|_{\SArm}\leq1$. Let $h(z,w)=1-f(z)\overline{f(w)}$ and let $R_h=1\otimes 1-f\otimes f$. Let 
\[
\CCC=\left\{A+\sum_{i=1}^d\left(B_i-M_{z_i}B_iM_{z_i}^\ast\right)~:~A, B_i\in\SSS^1(H^2(\DD^d)),~A, B_i\geq0\right\} \subset \SSS^1(H^2(\DD^d)).
\]
Then it is easy to verify that
\[
\LLL_c=\left\{L\in\BBB(H^2(\DD^d))~:~\Tr(LC)\geq0,~\forall C\in\CCC\right\}.
\]
Meanwhile, $\|f\|_{\SArm}\leq1$ implies
\[
\Tr(L_{\Tbf,\xi}R_h)= \la L_{\Tbf,\xi} 1,1 \ra - \la L_{\Tbf,\xi} f,f \ra = \|\xi\|^2-\|f(\Tbf)\xi\|^2\geq0,
\]
for all  cyclic commuting tuples $(\Tbf,\xi)$ of strict contractions.
As in the proof of Lemma \ref{lem: SA norm with Lc}, one checks that $\{L_{\Tbf,\xi}~:~(\Tbf,\xi)\text{ is a cyclic commuting tuple of strict contractions}\}$ is weak-$*$ dense in $\LLL_c$.
Therefore
\[
\Tr(LR_h)\geq0,\quad\forall L\in\LLL_c.
\]
By the Hahn-Banach cone separation theorem, $R_h$ is in the closure of $\CCC$. Note that for $C=A+\sum_{i=1}^d\left(B_i-M_{z_i}B_iM_{z_i}^\ast\right)\in\CCC$,
\begin{flalign*}
\la CK_w, K_z\ra=&\la AK_w, K_z\ra+\sum_{i=1}^d(1-z_i\bw_i)\la B_iK_w,K_z\ra\\
=&(1-z_1\bw_1)\left(\frac{\la AK_w, K_z\ra}{1-z_1\bw_1}+\la B_1K_w,K_z\ra\right)+\sum_{i=2}^d(1-z_i\bw_i)\la B_iK_w,K_z\ra.
\end{flalign*}
The last expression is an Agler decomposition. Suppose $\{C_n\}\subset\CCC$ and $C_n$ converges to $R_h$. Let
\[
h_n(z,w)=\sum_{i=1}^d(1-z_i\bw_i)k_{n,i}(z,w)
\]
be the Agler decomposition corresponding to $C_n$. Then $h(z,w)=\la R_hK_w,K_z\ra=\lim_{n\to\infty}h_n(z,w)$. Meanwhile, notice that by positivity, 
\[
|k_{n,i}(z,w)|^2\leq|k_{n,i}(z,z)||k_{n,i}(w,w)|\leq\frac{1}{(1-|z_i|^2)(1-|w_i|^2)}|\la C_nK_z,K_z\ra\la C_nK_w,K_w\ra|.
\]
Thus the functions $\{k_{n,i}(z,w)\}$ are locally uniformly bounded and holomorphic in $z$ and conjugate holomorphic in $w$.
The conclusion then follows from a normal family argument. This completes the proof.
\end{proof}

One way to explain the difference between our approach and Agler's is the following: we replace the duality pairing $\mathrm{Her}(\DD^d)\times\left(\mathrm{Her}(\DD^d)\right)^*$ (see Remark \ref{rem: two pairings}) with the more familiar pairing $\SSS^1(H^2(\DD^d))\times\BBB(H^2(\DD^d))$, where $\SSS^1(H^2(\DD^d))$ denotes the space of trace class operators, and the duality is given by the trace. More explicitly, assume $(\Tbf,\xi)$ is a cyclic commuting tuple of strict contractions, then
\[
(h,~\Lambda_{\Tbf,\xi})=\la h(\Tbf)\xi,\xi\ra=\sum_j\la p_j(\Tbf)\xi, q_j(\Tbf)\xi\ra=\sum_j\la L_{\Tbf,\xi}p_j,q_j\ra_2=\Tr(L_{\Tbf,\xi}R_h).
\]
Here the notation $(\cdot,\cdot)$ on the left denotes the pairing $\mathrm{Her}(\DD^d)\times\left(\mathrm{Her}(\DD^d)\right)^*$, and on the right we have the pairing $\SSS^1(H^2(\DD^d))\times\BBB(H^2(\DD^d))$. The cost is that not all hereditary functions define $R_h$, and not all cyclic commuting tuples define $L_{\Tbf,\xi}$. But the gain is that the pair $\SSS^1(H^2(\DD^d))\times\BBB(H^2(\DD^d))$ is equipped with richer structure. For many problems, the functions and operator tuples we lose are not essential for the argument. 

Taking the above one step further, the proof above of Theorem \ref{thm: Agler decomp} carries over when one replaces $\SSS^1(H^2(\DD^d))\times\BBB(H^2(\DD^d))$ with $\SSS^2(H^2(\DD^d))$ paired with itself. As was shown by techniques such as the lurking isometry argument, the Hilbert space structure can be important in solving sum-of-squares type problems. What the Hilbert structure of $\SSS^2(H^2(\DD^d))$ brings to this problem is yet to be discovered.

\section{The Dual Schur--Agler Norm}\label{sec: dual SA}
In Section \ref{sec: conv descrip of SA norm}, we gave convex descriptions of the Schur--Agler norm and the supremum norm. In this section, we show how these convex descriptions give descriptions of the pre-duals of these spaces. 
The following more general principle will be useful.

\begin{lem}
  \label{lem:cone_norm}
  Let $\mathcal{H}$ be a Hilbert space and let $\mathcal{L} \subset B(\mathcal{H})$ be a convex cone of positive operators
  containing the identity. Let $\xi \in \mathcal{H}$. For $g \in \mathcal{H}$, define
  \begin{equation*}
    \|g\|_* = \inf \{ \langle L \xi,\xi \rangle^{1/2} : L \in \mathcal{L}, L \ge g \otimes g \}.
  \end{equation*}
  Then:
  \begin{enumerate}[label=\normalfont{(\alph*)}]
    \item 
  $\|\cdot\|_*$ is a semi-norm on $\mathcal{H}$.
\item For all $f \in \mathcal{H}$,
  \begin{equation*}
    \sup \{ |\langle f,g \rangle|: g \in \mathcal{H}, \|g\|_* \le 1 \}
    = \sup \{ \langle L f ,f \rangle^{1/2}: L \in \mathcal{L}, \langle L \xi,\xi \rangle \le 1 \}.
  \end{equation*}
  
  \end{enumerate}
\end{lem}

\begin{proof}
  (a)
  Since $\mathcal{L}$ contains the identity, $\|g\|_* < \infty$ for all $g \in \mathcal{H}$. Absolute homogeneity of $\|\cdot\|_*$
  is a direct consequence of the assumption that $\mathcal{L}$ is a cone. We have to prove the triangle inequality.

  Let $g_1, g_2 \in \mathcal{H}$ and let $L_1, L_2 \in \mathcal{L}$ with $L_i \ge g_i \otimes g_i, i=1, 2$.
  Using the basic inequality $2 a b \le (t a)^2 + (b/t)^2$ for $a,b \ge 0$ and $t > 0$, we find that
  \begin{equation*}
    (g_1 + g_2) \otimes (g_1 + g_2) \le (1+t^2) g_1 \otimes g_1 + (1 + 1/t^2) g_2 \otimes g_2 \quad \text{ for all } t > 0.
  \end{equation*}
  Thus, defining $L = (1+t^2) L_1 + (1 + 1/t^2) L_2$, we have $L\geq (g_1 + g_2) \otimes (g_1 + g_2)$.
  Moreover, $L \in \mathcal{L}$ since $\mathcal{L}$ is a convex cone.
  It follows that
  \begin{equation*}
    \|g_1 + g_2\|_*^2 \le \langle L \xi, \xi \rangle = (1+t^2) \langle L_1 \xi, \xi \rangle + (1 + 1/t^2) \langle L_2 \xi, \xi \rangle
    \quad \text{ for all } t > 0.
  \end{equation*}
  Taking the infimum over $L_1$ and $L_2$ gives
  \begin{equation*}
    \|g_1 + g_2\|_*^2 \le (1+t^2) \|g_1\|_*^2 + (1 + 1/t^2) \|g_2\|_*^2
    \quad \text{ for all } t > 0.
  \end{equation*}
  If $\|g_1\|_*, \|g_2\|_* \neq 0$, then the right-hand side is minimized by taking $t = \sqrt{ \|g_2\|_*/\|g_1\|_* }$, in which case we obtain
  \begin{equation*}
    \|g_1 + g_2\|_*^2 \le (\|g_1\|_* + \|g_2\|_*)^2.
  \end{equation*}
  If $\|g_1\|_* = 0$ or $\|g_2\|_* = 0$, then we obtain the same inequality by taking the limit $t \to \infty$ (respectively $t \to 0)$.

  (b) Let $f \in \mathcal{H}$ and set $M = \sup \{ \langle L f, f \rangle^{1/2}: L \in \mathcal{L}, \langle L \xi,\xi \rangle \le 1 \}$
  and assume that $M < \infty$.
  Let $g \in \mathcal{H}$ and let $L \in \mathcal{L}$ with $L \ge g \otimes g$. Then
  \begin{equation*}
    | \langle f,g \rangle|^2 = \langle (g \otimes g) f, f \rangle \le \langle L f ,f \rangle \le M^2 \langle L \xi,\xi \rangle.
  \end{equation*}
  Taking the infimum over $L$ gives $|\langle f,g \rangle| \le M \|g\|_*$, which proves the inequality ``$\le$''.

  Conversely, given $f \in \mathcal{H}$ and $L \in \mathcal{L}$ with $\langle L \xi, \xi \rangle \le 1$,
  write $L = A^* A$ and let $P$ be the orthogonal projection onto the one dimensional
  space $\mathbb{C} A p$. Then $A^* P A$ is a positive operator of rank at most $1$, so it is of the form
  $A^* P A = g \otimes g$ for some $g \in \mathcal{H}$. Then
  $g \otimes g \le L$ by construction, so $\|g\|_* \le 1$. Moreover,
  \begin{equation*}
    \langle L f,f \rangle = \|A f\|^2 = \langle A^* P A f, f \rangle = \langle (g \otimes g) f,f \rangle = | \langle f,g \rangle |^2.
  \end{equation*}
  This establishes the remaining inequality.
\end{proof}

In light of Lemmas \ref{lem: SA norm with Lc}, \ref{lem: SA by graded}, \ref{lem:sup_norm_Lt}, and Lemma \ref{lem:cone_norm} above, we make the following definitions.
\begin{defn}
\begin{enumerate}
    \item For $g\in H^2(\DD^d)$, define
    \[
    \LLL_c(g)=\left\{L\in\LLL_c~:~L\geq g\otimes g\right\},
    \]
    and define its \emph{dual Schur--Agler norm} to be
    \[
    \|g\|_\ast=\inf\left\{\sqrt{\la L1,1\ra_2}~:~L\in\LLL_c(g)\right\}.
    \]
    \item For any nonnegative integer $n$ and any homogeneous polynomial $q\in P_n$, define
    \[
    \LLL_c^{(n)}(q)=\left\{L=\sum_{k=0}^n L_k\in\LLL_c(q)\cap\LLL_c^{(n)}~:~L_n=q\otimes q\right\}.
    \]
    \item For $g\in H^2(\DD^d)$, define
    \[
\LLL_t(g)=\left\{L\in\LLL_t~:~L\geq g\otimes g\right\},\quad\text{and}\quad\|g\|_{H_*^\infty}=\inf\left\{\sqrt{\la L1,1\ra_2}~:~L\in\LLL_t(g)\right\}.
\]
\end{enumerate}
\end{defn}

\begin{lem}\label{lem: dual SA geq coefficients}
For $g(z)=\sum_\alpha a_\alpha z^\alpha\in H^2(\DD^d)$, we have
\[
\|g\|_\ast\geq\sup_\alpha|a_\alpha|.
\]
\end{lem}

\begin{proof}
Let $L \in \mathcal{L}$ with $L \ge g \otimes g$.
  By definition of $L$, we have
  \begin{equation*}
    M_{z^\alpha}^* (g \otimes g) M_{z^\alpha} \le M_{z^\alpha}^* L M_{z^\alpha} \le L,
  \end{equation*}
  hence
  \begin{equation*}
    |a_\alpha|^2 = | \langle g, z^\alpha \rangle|^2 = | \langle M_{z^\alpha}^* (g \otimes g) M_{z^\alpha} 1,1 \rangle|^2
    \le \langle L 1, 1 \rangle_2.
  \end{equation*}
  Taking the infimum over all admissible $L$ gives the desired inequality.
\end{proof}

The following theorem justifies the name of $\|\cdot\|_*$.
\begin{thm}
  \label{thm:dual_norm_general}
  \begin{enumerate}
      \item The dual Schur--Agler norm $\|\cdot\|_\ast$ is a norm, and for each  $f\in H^2(\mathbb{D}^d)$, we have
  \begin{equation}
    \label{eqn:SA_dual}
    \|f\|_{\SArm} = \sup \{ | \langle f,g \rangle_2|: g\in H^2(\mathbb{D}^d), \|g\|_* \leq 1 \}.
  \end{equation}
  In fact, the map
  \begin{equation*}
    \SA_d \to (H^2(\mathbb{D}^d), \|\cdot\|_*)^*, \quad f \mapsto \langle \cdot,f \rangle_2,
  \end{equation*}
  is a conjugate linear isometric isomorphism.
      \item If $q \in P_n$, then
  \begin{equation}\label{eqn: SA dual homogeneous}
    \|q\|_* = \inf \left\{ \sqrt{\langle L1,1 \rangle_2}: L \in \mathcal{L}_c^{(n)}, L \ge q \otimes q \right\}= \inf \left\{ \sqrt{\langle L 1,1 \rangle_2}: L \in \mathcal{L}_c^{(n)}(q) \right\}.
  \end{equation}
  For $p\in P_n$,
  \begin{equation}\label{eqn: SA norm by dual homogeneous}
    \|p\|_{\SArm}=\sup\left\{|\la p, q\ra_2|~:~q\in P_n,~\|q\|_\ast\leq1\right\}.
  \end{equation}
  Consequently, the map
  \[
  (P_n,~\|\cdot\|_{\SArm})\to(P_n, \|\cdot\|_*)^*,\quad p\mapsto\la \cdot, p\ra_2
  \]
  is also a conjugate linear isometric isomorphism.
  \end{enumerate}
\end{thm}

\begin{proof}
 For ease of notation, we will write $\langle \cdot,\cdot \rangle = \langle \cdot,\cdot \rangle_2$ throughout.
  It follows from Lemma \ref{lem:cone_norm} (a) and Lemma \ref{lem: dual SA geq coefficients} that $\|\cdot\|_*$ is a norm.
 To show \eqref{eqn:SA_dual}, recall from Lemma \ref{lem: SA norm with Lc} that
  \begin{equation*}
    \|f\|_{\SArm}^2 = \sup \{ \langle L f, f \rangle: L \in \mathcal{L}_c, \langle L 1, 1 \rangle \le 1 \}
  \end{equation*}
  for all $f \in H^2(\mathbb{D}^d)$.
  Thus, \eqref{eqn:SA_dual} follows from part (b) of Lemma \ref{lem:cone_norm}.

 To complete the proof of statement (1), it remains to show surjectivity the map in statement (1).
Note that $(H^2(\mathbb{D}^d),\|\cdot\|_2)$
  is contractively contained in $(H^2(\mathbb{D}^d), \|\cdot\|_*)$ (simply choose $L$ to be a scalar multiple of the identity
  in the definition of $\|\cdot\|_*$). So if $\varphi \in (H^2(\mathbb{D}^d),\|\cdot\|_*)^*$, then
  there exists $f \in H^2(\mathbb{D}^d)$ such that $\varphi(g) = \langle g,f \rangle$ for all $g \in H^2(\mathbb{D}^d)$.
  By \eqref{eqn:SA_dual}, we must have $f \in \SA_d$. This proves statement (1).

Let us temporarily define
  \begin{equation*}
    \|q\|_*' = 
    \inf \left\{ \sqrt{\langle L1,1 \rangle_2}: L \in \mathcal{L}_c^{(n)}, L \ge q \otimes q \right\}.
  \end{equation*}
  Then $\|q\|_* \le \|q\|_*'$, so with the help of Lemma \ref{lem:cone_norm} (a),
  we find that $\|\cdot\|_*'$ is a norm. Moreover, Lemma \ref{lem: SA by graded} and  Lemma \ref{lem:cone_norm} (b),
  applied to the space of polynomials of degree at most $n$ proves
  that
  \[ 
    \|p\|_{\SArm} = \sup \{ | \la p,q \ra_2|: q \in P_n, \|q\|_\ast' \le 1 \}.
  \]
  On the other hand, by statement (1), we have
  \begin{equation*}
    \|p\|_{\SArm} \ge \sup \{ |\langle p,q \rangle_2|: q \in P_n, \|q\|_* \le 1 \}.
  \end{equation*}
  Since $\|q\|_* \le \|q\|_*'$, we must have equality throughout.
  This proves \eqref{eqn: SA norm by dual homogeneous} and by duality
  also $\|q\|_* = \|q\|'_*$, i.e.\ the first equality of \eqref{eqn: SA dual homogeneous}.

  To see the second equality, note that if $n \ge 1$ and
  $L = \sum_{k=0}^n L_k$ with $L_k \in \mathcal{B}(P_k)$, then $L \in \mathcal{L}_c$ if and only if
  \begin{equation*}
    M_{z_i}^* L_k M_{z_i} \le L_{k-1} \quad \text{ for all } 1 \le k \le n, 1 \le i \le d.
  \end{equation*}
  So if $L \in \mathcal{L}_c^{(n)}$ with $L \ge q \otimes q$,
  then $L_n \ge q \otimes q$ and so
  \[
    M_{z_i}^* (q \otimes q) M_{z_i} \le M_{z_i}^* L_n M_{z_i} \le L_{n-1}.
  \]
  This shows that $L' := q \otimes q + \sum_{k=0}^{n-1} L_k  \in \mathcal{L}_c$ as well,
  and $\langle L 1,1 \rangle_2 = \langle L_0 1,1 \rangle_2 = \langle L' 1,1 \rangle_2$.
  This shows the second equality, which is trivial if $n=0$.
  This completes the proof.
\end{proof}

The normed space $(H^2(\mathbb D^d),\|\cdot\|_*)$ is not complete, since $\|\cdot\|_* \le \|\cdot\|_{2}$, but the two norms are not equivalent. For instance, since $\|\cdot\|_{\infty} \le \|\cdot\|_{\SArm}$, we see that for each reproducing kernel $K_z$,
\[
    \|K_z\|_* = \sup \{ | \la f,K_z \ra_2|: \|f\|_{\SArm} \le 1\} = 1.
\]

We let $(\SA_d)_*$ be the completion of $(H^2(\mathbb D^d),\|\cdot\|_*)$ and continue
to denote the norm with $\|\cdot\|_*$. From Theorem \ref{thm:dual_norm_general}, we see that
the Cauchy pairing extends to a pairing between functions $f \in (\SA_d)_*$ and $g \in \SA_d$, and that this pairing gives a conjugate linear isometric isomorphism between $\SA_d$ and $((\SA_d)_*)^*$.
In particular, we can equip $\SA_d$ with the weak-$*$ topology given by this duality.
We will now record a few elementary properties of the Banach space $(\SA_d)_*$.

\begin{prop}\label{prop: dual SA basics}
  \begin{enumerate}[label=\normalfont{(\alph*)}]
    \item $(\SA_d)_*$ is a Banach space of analytic functions on $\mathbb D^d$ with continuous point evaluations.
    \item The space of polynomials $\mathbb C[\mathbf{z}]$ and $\mathrm{span} \{K_z: z \in \mathbb D^d \}$ are dense in $(\SA_d)_*$.
    \item The weak-$*$ topology on $\SA_d$ given by the duality with $(\SA_d)_*$ is the unique weak-$*$ topology on $\SA_d$ in which evaluation at every point in $\mathbb D^d$ is weak-$*$ continuous.
    On bounded subsets of $\SA_d$, the weak-$*$ topology agrees with the topology of pointwise convergence on $\mathbb D^d$.
    \item The space of polynomials $\mathbb C[\mathbf{z}]$ is weak-$*$ dense in $\SA_d$.
\end{enumerate}
\end{prop}

\begin{proof}
    We prove the statements in different order.
    
    (b) This follows since both spaces are dense in $(H^2(\mathbb D^d),\|\cdot\|_2)$
    and  $(H^2(\mathbb D^d),\|\cdot\|_2)$ is densely and contractively contained in $(\SA_d)_*$.

    (c) Since $K_z \in (\SA_d)_*$ for all $z \in \mathbb D^d$, point evaluations
    are weak-$*$ continuous. Moreover, by the density of the linear span
    of kernel functions in $(\SA_d)_*$, a bounded net converges weak-$*$ if and only if it converges pointwise on $\mathbb D^d$.
    
    The uniqueness statement is a general principle about weak-$*$ topologies
    on function spaces: If $\tau_1$ is the weak-$*$ topology given by $(\SA_d)_*$ and $\tau_2$
    is another weak-$*$ topology with continuous point evaluations, then the identity mapping
    $(\SA_d, \tau_2) \to (\SA_d,\tau_1)$ is continuous on the unit ball.
    Since the unit ball is $\tau_2$-compact and $\tau_1$ is Hausdorff, the identity mapping
    is a homeomorphism between the unit balls.
    Then the Krein--Smulian theorem implies that the identity mapping
    is a homeomorphism on all of $\SA_d$, i.e.\ $\tau_2 = \tau_1$.

    (d) Let $f \in \SA_d$ and let $p_n$ be the sequence of Fej\'er means of $f$. By Lemma \ref{lem:SA_approx}, we have $\|p_n\|_{\SArm} \le \|f\|_{\SArm}$ for all $n \in \mathbb N$ and $(p_n)$
    converges to $f$ pointwise on $\mathbb D^d$. By (c), this implies that $(p_n)$ converges to $f$ weak-$*$.

    (a)
    Let $c_0$ be the space of analytic functions on $\mathbb D^d$ whose Taylor coefficients
    converge to $0$, equipped with the supremum norm of the coefficients.
    By Lemma \ref{lem: dual SA geq coefficients}, $(H^2(\mathbb D^d),\|\cdot\|_*)$ is contractively contained in $c_0$, hence we obtain a contractive map
    $\Phi: (\SA_d)_* \to c_0$ that extends the inclusion. We claim that $\Phi$ is injective.

    By duality, it suffices to show that the adjoint $\Phi^*: \ell^1 \to \SA_d$ has weak-$*$ dense range.
    But this follows from weak-$*$ density of the polynomials in $\SA_d$, i.e.\ part (d). Thus, $\Phi$ is injective,
    so $(\SA_d)_*$ is a space of analytic functions contractively contained in $c_0$, which also
    implies continuity of point evaluations.
\end{proof}

Classically,  $H^\infty(\mathbb{D}^d)$ is the dual space of $L^1(\mathbb{T}^d) / (H^\infty(\mathbb{D}^d)_\bot$,
where the duality is given by the Cauchy pairing.
For comparison with the previous result, we record a different expression for the (pre-)dual $H^\infty$ norm.
\begin{prop}\label{prop: dual sup norm}
The function $\|\cdot\|_{H^\infty_*}$ is a norm, and for each $f \in H^2(\mathbb{D}^d)$, we have
  \begin{equation*}
    \|f\|_\infty = \sup \{ |\langle f,g \rangle_2|: g \in H^2(\mathbb{D}^d), \|g\|_{H^\infty_*} \leq 1 \}.
  \end{equation*}
\end{prop}

\begin{proof}
  The fact that $\|\cdot\|_{H^\infty_*}$ is a semi-norm follows from Lemma \ref{lem:cone_norm} (a).
  Since $\mathcal{L}_t \subset \mathcal{L}_c$, we have $\|\cdot\|_{H^\infty_*} \ge \|\cdot\|_*$, so it is a norm.
  The duality relation follows from Lemma \ref{lem:cone_norm} (b) and Lemma \ref{lem:sup_norm_Lt}.
\end{proof}

\begin{exam}
  Essentially the same operator $L$ as in Example \ref{exa:TTO} shows that the analogue of Theorem \ref{thm:dual_norm_general} (2) for $\|\cdot\|_{H^\infty_*}$ fails, i.e.\ if $q \in P_n$, then
  \begin{equation}
    \label{eqn:dual_TTO}
    \inf \left\{ \langle L1,1 \rangle^{1/2}: L = \sum_{k=1}^n L_k : L_k \text{ satisfy } \eqref{eqn:TTO}, L \ge q \otimes q \right\}
  \end{equation}
  can be strictly smaller than $\|q\|_{H^\infty_*}$. Indeed, let
  \begin{equation*}
    q(z) = z_1^2 + z_2^2 + z_3^2 -\frac{1}{2} (z_1 z_2 + z_2 z_3 + z_1 z_3)
  \end{equation*}
  be the polynomial of the Holbrook example. If $p$ is the Kaijser--Varopoulos polynomial, then $\|p\|_\infty = 5$
  and $\langle p,q \rangle_2 = 6$, so $\|q\|_{H^\infty_*} \ge \frac{6}{5}$.
  But if $L$ is the operator of Example \ref{exa:TTO} and $L' = \frac{5}{4} L$, then one checks that $L' \ge q \otimes q$ (the difference
  $L' - q \otimes q$ is a positive scalar multiple of a rank $2$ projection). Thus, the infimum in \eqref{eqn:dual_TTO}
  is at most $\sqrt{\frac{5}{4}} < \frac{6}{5}$.
\end{exam}

\section{Methods of Constructing $L$}\label{sec: methods of constructing L}
For a homogeneous polynomial $q\in P_n$, we provide several methods of constructing operators $L\in\LLL_c^{(n)}(q)$. This leads to upper bounds of the dual Schur--Agler norm $\|q\|_\ast$, which, by Theorem \ref{thm:dual_norm_general}, results in lower bounds of the Schur--Agler norm. We show how our methods recover  the counterexamples constructed in \cite{CrDa75}\cite{Dix76}\cite{Hol01}\cite{Varo74}; see the Appendix for a detailed explanation in terms of their $L$ operators. Finally, we give explicit computations of the Schur--Agler norms and dual Schur--Agler norms of Kaijser-Varopoulos-Holbrook type polynomials.
In the next section, we will focus on two of the resulting lower bounds, which are in the form of weak products. 

The Hankel operators play an important role in our construction. Recall the following important identity:
\begin{equation}
M_{z_i}^\ast\Gamma_q=\Gamma_qM_{z_i},\quad\forall i=1,\cdots,d.
\end{equation}

For $q\in P_n$ and $0\leq k\leq n-1$, let us define
\[
A_k(q)=\left\|\Gamma_q\big|_{P_k}\right\|;\quad B_k(q)=\max_{1\leq i\leq d}\left\|M_{z_i}^\ast\Gamma_q\big|_{P_k}\right\|.
\]
\begin{description}
    \item[Method 1:] for $0\leq k\leq n-1$, define
    \[
    L_1=\Gamma_{\widehat q}^\ast\Gamma_{\widehat q}\big|_{P_{\geq k+1}}+A_k(q)^2\cdot P_{\leq k};
    \]
    \item[Method 2:] for $0\leq k\leq n-1$, define
    \[
    L_2=\Gamma_{\widehat q}^\ast\Gamma_{\widehat q}\big|_{P_{\geq k+1}}+B_k(q)^2\cdot P_{\leq k};
    \]
\end{description}

\begin{lem}\label{lem: method 1 2}
Let $L_1, L_2$ be as above. Then $L_1, L_2\in\LLL_c^{(n)}(q)$. As a consequence, 
\[
  \|q\|_\ast\leq \min_{0 \le k \le n-1} B_k(q)\leq \min_{0 \le k \le n-1} A_k(q).
\]
\end{lem}
\begin{proof}
It is clear that each $L_i$ is positive and reduces $P_l$ for any $l$. An elementary computation (cf.\ \eqref{eqn:Gamma_star_Gamma} in the appendix) shows that $L_i|_{P_n}=q\otimes q$ and $L_i|_{P_{\le n}^\perp}=0$. Thus it remains to show the inequality $M_{z_j}^\ast L_i M_{z_j}\leq L_i$.
\begin{flalign*}
  M_{z_j}^\ast L_1 M_{z_j}=&M_{z_j}^\ast\left(\Gamma_{\widehat q}^\ast\Gamma_{\widehat q}\big|_{P_{\geq k+1}}+A_k(q)^2\cdot P_{\leq k}\right)M_{z_j}\\
  =&M_{z_j}^\ast\Gamma_{\widehat q}^\ast\Gamma_{\widehat q} M_{z_j}\big|_{P_{\geq k}}+A_k(q)^2\cdot P_{\leq k-1}\\
  =&\Gamma_{\widehat q}^\ast M_{z_j}M_{z_j}^\ast\Gamma_{\widehat q}\big|_{P_{\geq k}}+A_k(q)^2\cdot P_{\leq k-1}\\
  =&\Gamma_{\widehat q}^\ast M_{z_j}M_{z_j}^\ast\Gamma_{\widehat q}\big|_{P_{\geq k+1}}+\Gamma_{\widehat q}^\ast M_{z_j}M_{z_j}^\ast\Gamma_{\widehat q}\big|_{P_k}+A_k(q)^2\cdot P_{\leq k-1}\\
  \leq&\Gamma_{\widehat q}^\ast\Gamma_{\widehat q}\big|_{P_{\geq k+1}}+\left\|\Gamma_{\widehat{q}}\big|_{P_k}\right\|^2\cdot P_k+A_k(q)^2\cdot P_{\leq k-1}\\
=& L_1.
\end{flalign*}
Here, we used that $\| \Gamma_{\widehat{q}} \big|_{P_k}\| = \| \Gamma_q \big|_{P_k}\|=A_k(q)$.
A similar argument shows that $M_{z_j}^\ast L_2 M_{z_j}\leq L_2$. This completes the proof.
\end{proof}

We recall two definitions from Section \ref{sec: pre}.

\begin{defn}\label{defn: weak product and two norms}
Define
\[
\|p\|_{P_k\odot P_l}=\inf\left\{\sum_{i=1}^m\|f_i\|_2\cdot\|g_i\|_2~:~f_i\in P_k,~g_i\in P_l,\text{ and }p=\sum_{i=1}^m f_ig_i\right\},\quad\forall p\in P_{k+l},
\]
and
\[
\|p\|_{Z\odot P_k\odot P_l}=\inf\left\{\sum_{i=1}^d\|f_i\|_{P_k\odot P_l}~:~f_i\in P_{k+l},\text{ and }p=\sum_iz_if_i\right\},\quad\forall p\in P_{k+l+1}.
\]
For any $p\in P_n$, let us also define
\begin{equation*}
\vertiii{p}_1:=\max_{0\leq k\leq n}\|p\|_{P_k\odot P_{n-k}},
\end{equation*}
and
\begin{equation*}
\vertiii{p}_2:=\max_{0\leq k\leq n-1}\|p\|_{Z\odot P_k\odot P_{n-k-1}}.
\end{equation*}
\end{defn}

\begin{thm}\label{thm: norm 1 and 2}
For any $p\in P_n$,
\begin{equation}\label{eqn: SA lower bound}
\|p\|_{\SArm}\geq\vertiii{p}_2\geq\vertiii{p}_1.
\end{equation}
\end{thm}
\begin{proof}
  By Theorem \ref{thm:dual_norm_general}, Lemma \ref{lem: method 1 2} and Lemma \ref{lem:Hankel_WP}, we have
  \begin{equation*}
    \|p\|_{\SArm} = \sup_{q \in P_n \setminus \{0\}} \frac{| \langle p,q \rangle_2|}{\|q\|_*}
    \ge  \max_{0 \le k \le n-1} \sup_{q \in P_n \setminus \{0\}} \frac{| \langle p,q \rangle_2|}{B_k(q)}
    = \vertiii{p}_2.
  \end{equation*}
  The second inequality follows from \eqref{eqn:WP_easy_estimate} and the observation that $\|p\|_{P_k \odot P_{n-k}} = \|p\|_{P_{n-k} \odot P_k}$.
\end{proof}

As we will show in Section \ref{sec: weak product norm}, Method 1 is a simplification of Method 2, which sacrifices a constant depending only on $d$: $\vertiii{p}_1\approx_d\vertiii{p}_2$.
From the interpretation in the Appendix, the $L$ operators corresponding to \cite{CrDa75}\cite{Dix76}\cite{Varo74} match Method 2. In particular, the lower bound $\vertiii{\cdot}_2$ is bigger than  Dixon's lower bound in \cite{Dix76}. The weak product expressions of $\vertiii{\cdot}_1$ and $\vertiii{\cdot}_2$ also matches the general heuristic that polynomials with large Schur--Agler norms should be difficult to factor. 

However, we will show in Section \ref{sec: weak product norm} that the two lower bounds above are still far from optimal: for fixed $d$, $\vertiii{\cdot}_1$ and $\vertiii{\cdot}_2$ are comparable to the Hardy space norm $\|\cdot\|_2$.

Below we list a few more methods of constructing $L$ operators. Assume $q\in P_n$.
\begin{description}
    \item[Method 3:] For $0\leq k\leq n-1$, define
    \[
      C_k(q)=\max_{1\leq i\leq d}\left\|M_{z_i}^\ast|_{\Gamma_{\widehat q}(P_k)}\right\|=\max_{1\leq i\leq d}\left\|M_{z_i}^\ast|_{\Gamma_{q}(P_k)}\right\|,
    \]
    and
    \[
    L_3=q\otimes q+\sum_{k=0}^{n-1}C_k(q)^2\cdots C_{n-1}(q)^2\Gamma_{\widehat{q}}^\ast\Gamma_{\widehat{q}}\big|_{P_k}.
    \]
    Then $L_3\in\LLL_c^{(n)}(q)$. This gives the estimate
    \[
    \|q\|_\ast\leq \|q\|_2\cdot\prod_{k=0}^{n-1}C_k(q).
    \]
    Note that each $C_k(q)\leq1$. Also, in Remark \ref{rem: M3 for reproducing kernel} we show that Method 3 recovers the dual Schur-Agler norm of the reproducing kernels, while the previous two methods do not. Consequently, Method 3 results in a lower bound of $\|\cdot\|_{SA}$ bigger than $\|\cdot\|_\infty$.
    \item[Method 4:] One may combine Methods 2 and 3. Namely, for $0\leq k\leq n-1$, define
    \[
      L_4=q\otimes q+\sum_{j=k+1}^{n-1}C_j(q)^2\cdots C_{n-1}(q)^2\Gamma_{\widehat q}^\ast\Gamma_{\widehat q}\big|_{P_j}+B_k(q)^2C_{k+1}(q)^2\cdots C_{n-1}(q)^2\cdot P_{\leq k}.
    \]
     Then $L_4\in\LLL_c^{(n)}(q)$, from which we have
     \[
     \|q\|_\ast\leq B_k(q)C_{k+1}(q)\cdots C_{n-1}(q),\quad 0\leq k\leq n-1.
     \]
     \item[Method 5:] Method 3 can be generalized. Assume $E_k\in\BBB(P_k), 0\leq k\leq n$ are positive such that
    \[
      E_0=I,\quad M_{z_i}E_kM_{z_i}^\ast \Big|_{\Gamma_{\widehat{q}}(P_{n-k-1})} \leq E_{k+1}.
    \]
    Then 
    \[
      L_5=\sum_{k=0}^n\Gamma_{\widehat{q}}^\ast E_{n-k}\Gamma_{\widehat q}\big|_{P_k}\in\LLL_c^{(n)}(q).
    \]
    This results in the estimate
    \[
    \|q\|_\ast\leq\sqrt{\la E_n\widehat{q},\widehat{q}\ra_2}.
    \]
\end{description}
Likewise, one may also combine Methods 2 and 5. We leave the details to the interested reader.

As explained previously, Method 2 recovers the counterexamples in \cite{CrDa75}\cite{Dix76}\cite{Varo74}. The Holbrook counterexample is more delicate. As we will explain in Example \ref{Exam: Hol}, the relevant dual polynomial is $q(z)=z_1^2+z_2^2+z_3^2-\frac{1}{2}(z_1z_2+z_2z_3+z_1z_3)$, which satisfies $\|q\|_\ast=1$.

We compute 
\[
\begin{tabular}{|c|c|c|}
\hline
k&0&1\\
\hline
$A_k(q)$  & $\frac{\sqrt{15}}{2}$ & $\frac{3}{2}$\\
\hline
$B_k(q)$   & $\sqrt{\frac{3}{2}}$ & $\sqrt{\frac{3}{2}}$ \\
\hline
$C_k(q)$ &$\sqrt{\frac{2}{5}}$& $\sqrt{\frac{2}{3}}$\\
\hline
\end{tabular}~,\quad\|q\|_2= \frac{\sqrt{15}}{2}.
\]
Thus the best lower bound given by Method 1 and 2 for $\|q\|_*$ is $\sqrt{\frac{3}{2}}$, which is larger than $\|q\|_\ast=1$. However, $C_0(q)C_1(q)\|q\|_2=B_0(q)C_1(q)=1$. Thus Methods 3 and 4 recover $\|q\|_\ast$. 
The key point which makes $C_1(q) < 1$ is that $\Gamma_q(P_1)$ is the $2$-dimensional space
of degree $1$ polynomials whose coefficients sum to zero.
We end this section with an example.

\begin{exam}[Kaijser-Varopoulos-Holbrook type polynomials]
We give an example for which the Schur--Agler norm and the dual Schur--Agler norm is explicitly computed. For $t\in\CC$, write
\[
p_t(z)=\sum_{i=1}^dz_i^2+\frac{t}{2}\sum_{i\neq j}z_iz_j.
\]
In \cite[Section 6]{GrKaDmWo13}, $\|p_t\|_{\SArm}$ was computed when $t\in\RR, t<0, d>1, \frac{1}{t}<\frac{2-d}{4}$. 
We show that
\begin{equation}\label{eqn: pt SA}
\|p_t\|_{\SArm}= d\max\Big\{\Big|1-\frac{t}{2} \Big|, \Big|1+\frac{t(d-1)}{2} \Big|\Big\},\quad\forall t\in\CC.
\end{equation}
and
\begin{equation}\label{eqn: pt dual}
\|p_t\|_\ast= \frac{(d-1)|t-1|+|(d-1)t+1|}{d},\quad\forall t\in\CC.
\end{equation}
If $d$ is even, then $\|p_t\|_\infty = \|p_t\|_{\SArm}$. Indeed, take $z = (1,1,\ldots,1)$
and $z= (1, -1,1,-1,\ldots,-1)$. As is shown by the Kaijser-Varopoulos example and the Holbrook example, it can happen that $\|p_t\|_\infty<\|p_t\|_{\SArm}$.
\begin{proof}
Write $J$ to be the $d\times d$ matrix with all diagonal terms equal to $0$ and all other terms equal to $1$. Then $J$ has eigenvalues $\lambda_1=-1$, $\lambda_2=d-1$, with corresponding eigenspaces
\[
E_1=\left\{x\in\CC^d~:~\sum_ix_i=0\right\},\quad E_2=\CC\begin{bmatrix}
    1\\1\\\vdots\\1
\end{bmatrix}.
\]
For any commuting tuple $\Tbf$ of contractions,
\[
p_t(\Tbf)=\begin{bmatrix}
    T_1&T_2&\cdots&T_d
\end{bmatrix}
\left((I_d+\frac{t}{2}J)\otimes I\right)
\begin{bmatrix}
    T_1\\T_2\\\vdots\\T_d
\end{bmatrix}.
\]
The column operator and the row operator on the two sides have norm $\leq\sqrt{d}$. Therefore $\|p(\Tbf)\|\leq d\|I_d+\frac{t}{2}J\|=d\max\{|1-\frac{t}{2}|, |1+\frac{t(d-1)}{2}|\}$. Taking the supremum over $\Tbf$ gives
\begin{equation}\label{eqn: pt SA upper}
    \|p_t\|_{\SArm}\leq d\max\Big\{\Big|1-\frac{t}{2} \Big|, \Big|1+\frac{t(d-1)}{2} \Big|\Big\}.
\end{equation}
By the duality in Theorem \ref{thm:dual_norm_general},
\begin{align*}
\|p_t\|_\ast\geq&\sup_{s\in\CC}\frac{|\la p_t, p_s\ra_2|}{\|p_s\|_{\SArm}}\geq\sup_{s\in\CC}\frac{|1+\frac{(d-1)}{2}t\bar{s}|}{\max\{|1-\frac{s}{2}|, |1+\frac{s(d-1)}{2}|\}}\nonumber\\
=&\sup_{s\in\CC}\frac{\left|\la(1-\frac{s}{2}, 1+\frac{s(d-1)}{2}), (\frac{(d-1)(1-t)}{d}, \frac{1+(d-1)t}{d})\ra_2\right|}{\|(1-\frac{s}{2}, 1+\frac{s(d-1)}{2})\|_\infty}.
\end{align*}
Note that if $\alpha \in \mathbb{C}$ with $|\alpha| = 1$, then the equation
$1 - \frac{s}{2} = \alpha ( 1 + \frac{s (d-1)}{2})$ has a solution provided that $\alpha(d-1) \neq - 1$.
It follows that the last supremum is the $1$-norm of the second vector in the inner product, so
\begin{equation}
  \label{eqn: pt dual lower}
  \|p_t\|_\ast \ge 
\left\|\left(\frac{(d-1)(1-t)}{d}, \frac{1+(d-1)t}{d}\right)\right\|_1
=\frac{(d-1)|1-t|+|1+(d-1)t|}{d}.
\end{equation}

On the other hand, let $x>0, y\in\RR$ be determined later. Let
\[
L=p_t\otimes p_t+L_1+x,\quad\text{where }L_1=xI_d+y J\in\BBB(P_1).
\]
Here we take the basis $\{z_i\}_{i=1}^d$ of $P_1$ and identify $L_1$ with its matrix expression. Since $M_{z_i}^* J M_{z_i} = 0$, the condition $L\in\LLL_c^{(2)}(p_t)$ is equivalent to:
\begin{equation*}
L_1\geq v_i v_i^\ast, \quad i=1,\cdots,d,
\end{equation*}
where $v_i$ is the vector representation of $M_{z_i}^\ast p_t$ in the basis $\{z_i\}$. By symmetry, this is equivalent to the single condition
\begin{equation}\label{eqn: temp KVH}
L_1\geq v_1 v_1^\ast.
\end{equation}
Note that $L_1$ has eigenvalues $\mu_1:= x- y$ and $\mu_2 := x + (d-1) y$.
We may temporarily increase $x$ a little and assume $L_1$ is invertible, i.e.\ $\mu_1,\mu_2 > 0$.
Then \eqref{eqn: temp KVH} is equivalent to
\begin{equation}\label{eqn: temp KVH 3}
  \langle L_1^{-1} v_1,v_1 \rangle \le 1 \quad\Leftrightarrow\quad
  \mu_2 |t-1|^2(d-1)+ \mu_1 \left|(d-1)t+1\right|^2 \le d \mu_1 \mu_2.
\end{equation}
It follows that without the invertibility assumption on $L_1$, \eqref{eqn: temp KVH} is equivalent to the second condition
in \eqref{eqn: temp KVH 3} and $\mu_1,\mu_2 \ge 0$.
Minimizing $x=\frac{(d-1)\mu_1+\mu_2}{d}$ under these restrictions gives
\[
x=\frac{\left((d-1)|t-1|+|(d-1)t+1|\right)^2}{d^2},\quad y=\frac{\left(|(d-1)t+1|-|t-1|\right)\left(|(d-1)t+1|+(d-1)|t-1|\right)}{d^2}.
\]
This shows that
\begin{equation}\label{eqn: pt dual upper}
\|p_t\|_\ast\leq \frac{(d-1)|t-1|+|(d-1)t+1|}{d}.
\end{equation}
Combining \eqref{eqn: pt dual lower} and \eqref{eqn: pt dual upper} gives \eqref{eqn: pt dual}. By duality,
we also have
\begin{equation*}
  \|p_t\|_{\SArm} \ge \sup_{s \in \mathbb{C}} \frac{| \langle p_t,p_s \rangle_2}{\|p_s\|_*}
  = \sup_{s \in \mathbb{C}} d^2 \frac{| 1 + \frac{(d-1)}{2} t \overline{s}|}{(d-1) |1-s| + | 1 +(d-1) s|}
 \ge d\max\Big\{\Big|1-\frac{t}{2} \Big|, \Big|1+\frac{t(d-1)}{2} \Big|\Big\},
\end{equation*}
which can be seen by choosing $s=1$ and $s = -\frac{1}{d-1}$.
In combination with \eqref{eqn: pt SA upper}, this gives \eqref{eqn: pt SA}.
\end{proof}
\end{exam}

\section{The Weak Product Norm}\label{sec: weak product norm}

In this section, we will study the weak product norm appearing on the right-hand side of \eqref{eqn: SA lower bound} in more detail. Our goal is to prove the following result.

\begin{thm}
\label{thm:three_norm_two_norm}
    Let  $d \in \mathbb N$. There exist constants $C_d \le \binom{2d - 2}{d-1}$
    such that for all $p \in P_{d,n}$, we have
  \begin{equation*}
    \|p\|_2 \le \vertiii{p}_1 \le \sqrt{C_d} \|p\|_2.
  \end{equation*}
  Moreover,
  \begin{equation*}
    \vertiii{p}_1\leq\vertiii{p}_2\leq \sqrt{d}\vertiii{p}_1.
  \end{equation*}
\end{thm}

Thus, both norms are equivalent to the Hardy norm.
In particular, Theorem \ref{thm: norm 1 and 2} alone cannot be used to show
that von Neumann's inequality does not hold up to a constant.

We start by proving the second inequality.
\begin{lem}\label{lem: two norms equivalent}
For $p\in P_{d,n}$, let $\vertiii{p}_1$ and $\vertiii{p}_2$ be defined as in Definition \ref{defn: weak product and two norms}. Then
\begin{equation*}
\vertiii{p}_1\leq\vertiii{p}_2\leq \sqrt{d}\vertiii{p}_1.
\end{equation*}
\end{lem}

\begin{proof}
  The first inequality was already observed in Theorem \ref{thm: norm 1 and 2}.
  For any $p\in P_{d,n}$ and any $0\leq k\leq n-1$, suppose $p=\sum_{j=1}^m f_jg_j$, where $f_j\in P_{d,k}, g_j\in P_{d,n-k}$. For each $j$, decompose $g_j$ into an orthogonal sum $g_j=\sum_{i=1}^d z_ig_{i,j}$, so $\|g_j\|^2 = \sum_{i} \|g_{ij}\|^2_2$. Write $h_i=\sum_{j=1}^mf_jg_{i,j}$. Then $p=\sum_{i=1}^d z_ih_i$, and
\[
\|h_i\|_{P_{d,k}\odot P_{d,n-k-1}}\leq\sum_{j=1}^m\|f_j\|_2\|g_{i,j}\|_2.
\]
Therefore
\[
\|p\|_{Z_d\odot P_{d,k}\odot P_{d,n-k-1}}\leq\sum_i\|h_i\|_{P_{d,k}\odot P_{d,n-k-1}}\leq \sum_{j=1}^m\|f_j\|_2\sum_{i=1}^d\|g_{i,j}\|_2\leq\sqrt{d}\sum_{j=1}^m\|f_j\|_2\|g_j\|_2.
\]
Taking the infimum over all decompositions of $p$ gives $\|p\|_{Z_d\odot P_{d,k}\odot P_{d,n-k-1}}\leq \sqrt{d}\|p\|_{P_{d,k}\odot P_{d,n-k}}$. This completes the proof.
\end{proof}

To prove the first inequality in Theorem \ref{thm:three_norm_two_norm}, we begin with a simple observation.

\begin{lem}
  \label{lem:factor_small_norm}
  If $\alpha$ is a multi-index, $q$ is a homogeneous polynomial and $|\alpha| \ge \deg(q)$, then
  \begin{equation*}
    \vertiii{z^\alpha q}_1 = \|z^\alpha q\|_2 = \|q\|_2.
  \end{equation*}
\end{lem}

\begin{proof}
  Choosing $k = 0$ in the definition of $\vertiii{\cdot}_1$ shows that $\vertiii{z^\alpha q}_1 \ge \|z^\alpha q\|_2$, and
  it is clear that $\|z^\alpha q \|_2 = \|q\|_2$. Conversely, let $n = |\alpha| + \deg(q)$ be the degree
  of $z^\alpha q$, and let $k \in \{0,1,\ldots,n \}$. We wish to estimate $\|z^\alpha q\|_{P_{n-k} \odot P_k}$.
  By replacing $k$ with $n-k$ if necessary, we may without loss of generality assume that $k \le \frac{n}{2}$.
  Thus, $|\alpha| \ge k$.  Write $\alpha = \beta + \gamma$ for multi-indices $\beta,\gamma$ with $|\beta| = k$.
  Then
  \begin{equation*}
    \|z^\alpha q\|_{P_{d,n-k}\odot P_{d,k}}\le \|z^\beta\|_2 \|z^\gamma q\|_2 = \|q\|_2.
  \end{equation*}
  Taking the maximum over all $k\in\{0,\cdots,n\}$ gives $\vertiii{z^\alpha q}_1\leq\|q\|_2$,
  which proves the remaining inequality.
\end{proof}

If $t = (t_1,\ldots,t_d), s= (s_1,\ldots,s_d) \in [0,1]^d$, we write
$t \le s$ to mean that $t_j \le s_j$  for all $j = 1,\ldots,d$.
We require the following combinatorial lemma.

\begin{lem}
  \label{lem:combinatorics_real}
  Let $d \in \mathbb{N}$, let
  \begin{equation*}
    A_d = \Big\{ (t_1,\ldots,t_d) \in [0,1]^d: \sum_{j} t_j = 1\Big\}
  \end{equation*}
  and let
  \begin{equation*}
    B_d = \Big\{  (s_1,\ldots,s_d) \in [0,1]^d: \sum_{j} s_j = \frac{1}{2}\Big\}.
  \end{equation*}
  There exists a finite subset $F_d \subset B_d$ such that for each $t \in A_d$,
  there exists $s \in F_d$ with $s \le t$. We can achieve $|F_d| \le \binom{2 d-2}{d-1}$.
\end{lem}

\begin{proof}
If $s \in B_d$, then
\[
    \{t \in A_d: t \ge s \} = s + B_d.
\]
Thus, a subset $F_d \subset B_d$ has the desired property if and only if
\[ 
    A_d \subset \bigcup_{s \in F_d} s + B_d.
\]
On the other hand, if $G_d \subset \mathbb R^d$ is any subset such that
\[
    A_d \subset \bigcup_{s \in G_d} s + B_d,
\]
then we also have
\[ 
    A_d \subset \bigcup_{s \in G_d \cap B_d} s + B_d.
\]
Therefore, it suffices to show that the simplex $A_d$ can be covered by at most $\binom{2 d - 2}{d-1}$ translates of the simplex $B_d = \frac{1}{2} A_d$. By applying an affine transformation, this is the same problem as covering the simplex
\[
    K = \{x \in \mathbb [0,\infty)^{d-1}: \sum_{j} x_j \le d-1 \}
    \]
    by translates of $\frac{1}{2} K$. A more general version of this problem was solved in \cite[Proposition 2.1]{XLZ21}, which in particular shows that $\binom{2d-2}{d-1}$ translates suffice.
\end{proof}

\begin{rem}
  \label{rem:size}
  By Stirling's formula, $\binom{2 d - 2}{d-1}$ grows like a constant times $d^{-1/2} 4^d$.
  On the other hand, comparing $d-1$-dimensional volumes of the simplices $A_d$ and $B_d = \frac{1}{2} A_d$ gives the lower bound $|F_d| \ge 2^{d-1}$.
  In particular, $|F_d|$ necessarily grows exponentially in $d$.
\end{rem}

We will write
\begin{equation*}
  X_{d,n} = \{\alpha \in \mathbb{N}_0^d: |\alpha| = n \}
\end{equation*}
for the set of multi-indices in $d$-variables of length $n$. If $r \in [0,\infty)$, we also define
\begin{equation*}
  X_{d, \ge r} = \{\alpha \in \mathbb{N}_0^d: |\alpha| \ge r \}.
\end{equation*}
The following is a discrete analogue of Lemma \ref{lem:combinatorics_real}.

\begin{lem}
  \label{lem:combinatorics_multi}
  Let $d \in \mathbb{N}$. There exists a constant $C_d \le \binom{2 d - 2}{d-1}$ such that for all
  $n \in \mathbb{N}_0$, there exists a subset $G \subset X_{d,\ge \frac{n}{2}}$ with $|G| \le C_d$
  such that for each $\alpha \in X_{d,n}$, there exists $\beta \in G$ with $\beta \le \alpha$.
\end{lem}

\begin{proof}
  The statement is trivial if $n=0$, so let $n \ge 1$.
  Let $A_d,B_d$ be the sets appearing in Lemma \ref{lem:combinatorics_real} and let $F_d \subset B_d$
  the finite set provided by the same lemma. We claim that we can take $C_d = |F_d|$.
  To see this, let $n \in \mathbb{N}_0$ and define
  \begin{equation*}
    G = \{ ( \lceil n s_1 \rceil, \ldots, \lceil n s_d \rceil): (s_1,\ldots,s_d) \in F_d \}.
  \end{equation*}
  Here $\lceil x\rceil$ denotes the ceiling of $x$.
  It is clear that $G \subset X_{d,\ge \frac{n}{2}}$. If $\alpha \in X_{d,n}$, let $t= \alpha / n$.
  Then $t \in A_d$, so by the defining property of $F_d$, there exists $s \in F_d$ with $s \le t$.
  Let $\beta = ( \lceil n s_1 \rceil, \ldots, \lceil n s_d \rceil) \in G$.
  Then
  \begin{equation*}
    n s \le n t = \alpha,
  \end{equation*}
  and since $\alpha$ has integer components, we conclude that $\beta \le \alpha$, as desired.
\end{proof}

We are now ready for the proof of the first inequality in Theorem \ref{thm:three_norm_two_norm}.
The basic idea to decompose a homogeneous polynomial into band-limited polynomials already
appeared in \cite[Lemma 3.3]{Hartz25}.

\begin{proof}[Proof of Theorem \ref{thm:three_norm_two_norm}]
    Let $p \in P_{d,n}$.
  The lower bound $\vertiii{p}_1 \ge \|p\|_2$  follows by choosing $k=0$ in the definition of $\vertiii{\cdot}_1$. For the upper bound, let $G \subset X_{d, \ge \frac{n}{2}}$ be as in the conclusion
  of Lemma \ref{lem:combinatorics_multi}.
  For each $\alpha \in X_{d,n}$, choose $\beta(\alpha) \in G$ with $\beta(\alpha) \le \alpha$, and let
  $\gamma(\alpha) = \alpha - \beta(\alpha)$.
  Write
  \begin{equation*}
    p(z) = \sum_{\alpha \in X_{d,n}} a_\alpha z^\alpha
    = \sum_{\alpha \in X_{d,n}} a_\alpha z^{\beta(\alpha) + \gamma(\alpha)}
    = \sum_{\beta \in G} z^\beta \sum_{\substack{\alpha \in X_{d,n} \\ \beta(\alpha) = \beta}} a_\alpha z^{\gamma(\alpha)}.
  \end{equation*}
  Let $q_\beta(z)$ denote the inner sum. Then
  \begin{equation*}
    p = \sum_{\beta \in G} z^\beta q_\beta \quad \text{ and } \quad
    \|p\|_2^2 = \sum_{\beta \in G} \|q_\beta\|_2^2.
  \end{equation*}
  Since $|\beta| \ge \frac{n}{2}$, Lemma \ref{lem:factor_small_norm} applies to show that
  \begin{equation*}
    \vertiii{p}_1 \le \sum_{\beta \in G} \vertiii{z^\beta q_\beta}_1 = \sum_{\beta \in G} \|q_\beta\|_2
    \le |G|^{1/2} \Big( \sum_{\beta \in G} \|q_\beta\|_2^2 \Big)^{1/2}
    \le C_d^{1/2} \|p\|_2,
  \end{equation*}
  where we have used the Cauchy--Schwarz inequality in the penultimate step.
\end{proof}

\begin{rem}\label{rem: M3 for reproducing kernel}
  By Theorem \ref{thm: norm 1 and 2}, we have the lower bounds $\|p\|_{\SArm} \ge \vertiii{p}_2 \ge \vertiii{p}_1$
  for $p \in P_n$. Whereas these lower bounds can exceed $\|p\|_\infty$, Theorem
  \ref{thm:three_norm_two_norm} shows that they will often be much smaller than $\|p\|_\infty$.
  This failure of $\vertiii{\cdot}_1$ and $\vertiii{\cdot}_2$ to capture the supremum norm is closely related
  to the following observation. Let $K_z^{(n)}(w)=\sum_{|\alpha|=n}\bz^\alpha w^\alpha$ be the reproducing kernel of $P_{d,n}$ at $z\in\TT^d$. One can check (cf.\ the discussion following Theorem \ref{thm:dual_norm_general}) that $\|K_z^{(n)}\|_\ast=1$. For $0\leq k\leq n-1$, let $a_k=\|K_z^{(k)}\|_2=\sqrt{k+d-1\choose d-1}$.
  For each $k$, the Hankel operator $\Gamma_{K_z^{(n)}}$ acts as a rank $1$ operator from $P_k$ to $P_{n-k}$. Using this, one checks that
\[
A_k(K_z^{(n)})=a_k a_{n-k},\quad\text{and}\quad B_k(K_z^{(n)})=a_k a_{n-k-1}.
\]
Clearly, each $A_k(K_z^{(n)})$ or $B_k(K_z^{(n)})$ alone is significantly bigger than $1$. Thus Method 1 and Method 2 fail to recover the dual Schur--Agler norm of $K_z^{(n)}$. 

One can also compute
\[
C_k(q)=\frac{a_{n-k-1}}{a_{n-k}},
\]
which implies
\[
C_0(K_z^{(n)})\cdots C_{n-1}(K_z^{(n)})\cdot\|K_z^{(n)}\|_2=\frac{a_0}{a_n}\cdot a_n=1=\|K_z^{n}\|_\ast,
\]
and
\[
C_{n-1}(K_z^{(n)})\cdots C_1(K_z^{(n)}) B_0(K_z^{(n)})=\frac{a_0}{a_{n-1}}\cdot a_0 a_{n-1}=a_0^2=1=\|K_z^{n}\|_\ast.
\]
Thus Methods 3 and 4 recover $\|K_z^{(n)}\|_\ast$. Thus, by duality, its resulting lower bound for $\|\cdot\|_{\SArm}$ is at least as large as the supremum norm. However, in order to know whether they are equivalent, more techniques need to be developed.
\end{rem}

\begin{rem}
    The upper bound for $\frac{\vertiii{p}_1}{\|p\|_\infty}$ in Theorem \ref{thm:three_norm_two_norm} is
    on the order of $2^d d^{-1/4}$. Conversely, the Dixon example (cf. Example \ref{exam: Dixon} and Remark \ref{rem: Dixon exponential growth}) shows
    that
    \[
        \sup \left\{ \frac{\vertiii{p}_2}{\|p\|_\infty}: p \in P_{d,n}, n \in \mathbb N \right\}
    \]
     grows exponentially in $d$, hence so does
    \[
        \sup \left\{ \frac{\vertiii{p}_1}{\|p\|_\infty}: p \in P_{d,n}, n \in \mathbb N \right\}.
    \]
    by Lemma \ref{lem: two norms equivalent}.
\end{rem}

\section{Semidefinite Programming}\label{sec: SDP}

The definition of the Schur--Agler norm of a homogeneous polynomial $p \in P_n$,
\begin{equation*}
  \|p\|_{\SArm} = \sup \{ \|p(\Tbf)\|: \Tbf \text{ tuple of commuting contractions} \}
\end{equation*}
does not appear to be amenable to numerical computations, because of the complicated
nature of the set of $d$-tuples of commuting contractions, see for instance \cite{HO01,Sivic12},
and the non-linearity of the map $\Tbf \mapsto p(\Tbf)$.

In contrast, Lemma \ref{lem: SA by graded} gives the expression
\begin{equation*}
  \|p\|_{\SArm}^2 = \sup \{ \langle L p,p \rangle_2: L \in \mathcal{L}_c^{(n)}, \langle L 1,1 \rangle_2 \le 1 \},
\end{equation*}
where the convex cone $\LLL_c^{(n)}$ is given by
\begin{equation*}
  \mathcal{L}_c^{(n)} = \left\{ L = \sum_{k=0}^n L_k, L_k \in B(P_k), L_k \ge 0, M_{z_i}^* L_k M_{z_i} \le L_{k-1} \text{ for } 1 \le k \le n, 1 \le i \le d \right\}.
\end{equation*}
Thus, $\|p\|_{\SArm}^2$ is expressed as the solution of a convex optimization problem.

In fact, the feasible region is described by affine and semi-definite constraints, and the objective function is linear. Such convex optimization problems are known as semi-definite programs (SDP),
which can often be solved efficiently in practice. Similarly, Theorem \ref{thm:dual_norm_general} (2) yields an SDP for computing the square of the dual Schur--Agler norm of a homogeneous polynomial.

A very basic implementation of this algorithm using the MATLAB toolbox Yalmip \cite{yalmip} with SDP solver MOSEK \cite{mosek} computes
the Schur--Agler norm of the Kaijser--Varopoulos polynomial with approximately $8$ digits
of accuracy; the runtime on a standard laptop computer is less than $0.2$ seconds.
Moreover, the algorithm returns a maximizing operator $L$, which can serve
as a numerical certificate for a lower bound of the Schur--Agler norm of $p$.

We have also tested the algorithm on some low degree homogeneous polynomials in $d=4$ variables
(the first case where it is not known if Schur--Agler norm and supremum norm are comparable).
For instance, in degree $6$, the run time is on the order of a few seconds per polynomial.
However, the runtime increases quickly with the degree, and it is known that the maximum
of the ratio $\frac{\|p\|_{\SArm}}{\|p\|_\infty}$ can at most grow logarithmically in the degree.
Moreover, whereas the computation of $\|p\|_{\SArm}$ is a convex optimization problem for a fixed polynomial, maximizing the ratio $\frac{\|p\|_{\SArm}}{\|p\|_\infty}$ over all degree $n$ polynomials $p$ is a non-convex optimization problem (since we are maximizing a function that is not concave) with many local maxima. Thus, it seems difficult to get a feeling for whether von Neumann's inequality for $4$-variable homogeneous polynomials holds up to a constant from such numerical experiments alone.

Finally, we remark that another way of realizing the square of the Schur--Agler norm of a homogeneous polynomial as an SDP can be based on the Caratheodory interpolation type result for the Schur--Agler norm of Knese; see \cite[Theorem 1.3]{knese25}. In this case, the program will return a certificate for an upper bound of the Schur--Agler norm (an Agler decomposition). The algorithm based on Lemma \ref{lem: SA by graded} appears to be faster in practice.

Readers who wish to experiment with this algorithm can find a basic implementation at \url{https://github.com/michaelhartz/Schur-Agler-norm}.

\section{Appendix: The classical  counterexamples}\label{sec: counterexamples}
In this appendix, we explain the counterexamples in \cite{CrDa75}\cite{Dix76}\cite{Hol01}\cite{Varo74} in terms of the $L_{\Tbf,\xi}$ operators. In order to link them to our methods in Section \ref{sec: methods of constructing L}, the following general
identity for $q \in P_{d,n}$, which can be verified by direct computation, is useful:
\begin{equation}
  \label{eqn:Gamma_star_Gamma}
  \Gamma_q^* \Gamma_q = \sum_{|\alpha| \le n} M_{z^\alpha}^* \widehat{q} \otimes M_{z^\alpha}^* \widehat{q}.
\end{equation}
\begin{exam}[{\bf Varopoulos-Kaijser 1974 \cite{Varo74}}]\label{exam: Varo74}~
\begin{description}
    \item[Original construction:] Take
\[
T_1=\begin{bmatrix}
        0&0&0&0&0\\
        1&0&0&0&0\\
        0&0&0&0&0\\
        0&0&0&0&0\\
        0&\frac{1}{\sqrt{3}}&-\frac{1}{\sqrt{3}}&-\frac{1}{\sqrt{3}}&0
    \end{bmatrix},
     T_2=\begin{bmatrix}
        0&0&0&0&0\\
        0&0&0&0&0\\
        1&0&0&0&0\\
        0&0&0&0&0\\
        0&-\frac{1}{\sqrt{3}}&\frac{1}{\sqrt{3}}&-\frac{1}{\sqrt{3}}&0
    \end{bmatrix},
    T_3=\begin{bmatrix}
        0&0&0&0&0\\
        0&0&0&0&0\\
        0&0&0&0&0\\
        1&0&0&0&0\\
        0&-\frac{1}{\sqrt{3}}&-\frac{1}{\sqrt{3}}&\frac{1}{\sqrt{3}}&0
    \end{bmatrix},
\]
and take $p=z_1^2+z_2^2+z_3^2-2z_1z_2-2z_1z_3-2z_2z_3.$ Then one verifies
\[
p(\Tbf)=\begin{bmatrix}
    0&0&0&0&0\\0&0&0&0&0\\0&0&0&0&0\\0&0&0&0&0\\3\sqrt{3}&0&0&0&0
\end{bmatrix}.
\]
So
\[
\|p(\Tbf)\|=3\sqrt{3}.
\]
Meanwhile, in \cite{Varo74}, it was shown that $\|p\|_\infty=5$, which is strictly less than $\|p(\Tbf)\|$.
\item[The $L$ operator:] We give the details of computing the $L$ operator in this example, and directly give the results in the rest of the examples. From the expression of $p(\Tbf)$, the obvious choice is $\xi=\begin{bmatrix}
    1&0&0&0&0
\end{bmatrix}^T$. By direct computation, we verify that
\[
\la \xi, \xi\ra=1,\quad\la T_i\xi,T_j\xi\ra=\delta_{i,j},
\]
\[
\la T_i^2\xi, T_jT_k \xi\ra=-\frac{1}{3}\text{ if }j\neq k,\quad\la T_i^2\xi,T_j^2\xi\ra=\frac{1}{3},\quad\la T_iT_j\xi, T_kT_l\xi\ra=\frac{1}{3}\text{ if }i\neq j, k\neq l,
\]
and $\la T^\alpha \xi, T^\beta \xi\ra=0$ for other cases. In other words,
\[
\la L_{\Tbf,\xi}1,1\ra=1,\quad\la L_{\Tbf,\xi}z_i,z_j\ra=\delta_{i,j},
\]
\[
\la L_{\Tbf,\xi}z_i^2, z_jz_k\ra=-\frac{1}{3}\text{ if }j\neq k,\quad\la L_{\Tbf,\xi}z_i^2, z_j^2\ra=\frac{1}{3},\quad\la L_{\Tbf,\xi}z_iz_j, z_kz_l\ra=\frac{1}{3}\text{ if }i\neq j, k\neq l,
\]
and $\la L_{\Tbf,\xi}z^\alpha,z^\beta\ra=0$ in other cases. 
This gives
\begin{flalign*}
L_{\Tbf,\xi}
=&1\otimes1+\sum_{i=1}^3z_i\otimes z_i+\frac{1}{3}q\otimes q,
\end{flalign*}
where $q(z)=z_1^2+z_2^2+z_3^2-z_1z_2-z_2z_3-z_1z_3$.
\item[Interpretation:] For $i=1, 2, 3$,
\[
M_{z_i}^\ast L_{\Tbf,\xi}M_{z_i}=1\otimes 1+\frac{1}{3}\left(M_{z_i}^\ast q\right)\otimes\left(M_{z_i}^\ast q\right).
\]
Since $\|M_{z_i}^*q\|_2^2={3}$, we have $\left\|\frac{1}{3}\left(M_{z_i}^\ast q\right)\otimes\left(M_{z_i}^\ast q\right)\right\|=1$. So $\frac{1}{3}\left(M_{z_i}^\ast q\right)\otimes\left(M_{z_i}^\ast q\right)$ is less than the projection operator on $P_1$, which is $\sum_{i=1}^3z_i\otimes z_i$. Therefore $L_{\Tbf,\xi}\in\LLL_c$. In fact, from its expression, $L_{\Tbf,\xi}\in\LLL_c^{(2)}(\frac{q}{\sqrt{3}})$. We can now explain this counterexample using Lemma \ref{lem: SA norm with Lc}: write $L=L_{\Tbf,\xi}$, then
\[
\|p\|_{\SArm}^2\geq\frac{\la Lp,p\ra_2}{\la L1,1\ra_2}=\frac{|\la p, q\ra_2|^2}{3}=27>25=\|p\|_\infty^2.
\]
We can also explain with Theorem \ref{thm:dual_norm_general}: 
\[
\|q\|_\ast\leq\sqrt{3}\quad\Rightarrow\quad\|p\|_{\SArm}\geq\frac{|\la p, q\ra_2|}{\|q\|_*}\geq3\sqrt{3}>5=\|p\|_\infty.
\]
(In fact, Example \ref{eqn: pt dual} shows that  $\|q\|_\ast = \frac{5}{3}$.)
Using Identity $\eqref{eqn:Gamma_star_Gamma}$, we can also write
\begin{equation*}
  L = \frac{1}{3} \Gamma_q^* \Gamma_q \big|_{P_2} + P_{\le 1}.
\end{equation*}
Since $\|M_{z_i}^* \Gamma_q \big|_{P_2}\|^2 = 3$ for $i=1,2,3$, we see that this example can also be obtained from Method 2 in Section \ref{sec: methods of constructing L}.

\end{description}
\end{exam}

\begin{exam}[{\bf Crabb-Davie 1975 \cite{CrDa75}}]\label{exam: CD}~
\begin{description}
    \item[Original construction:] $\Tbf=(T_1,T_2,T_3)$ is given by the following mapping between the standard orthonormal basis of $\CC^8$:
    \begin{flalign*}
    T_1: &e_1\to e_2\to(-e_5)\to(-e_8)\to0, e_3\to e_7\to0, e_4\to e_6\to0\\
    T_2: &e_1\to e_3\to(-e_6)\to(-e_8)\to0, e_2\to e_7\to0, e_4\to e_5\to0\\
    T_3: &e_1\to e_4\to(-e_7)\to(-e_8)\to0, e_2\to e_6\to0, e_3\to e_5\to0,
    \end{flalign*}
    and
    \[
    p=z_1z_2z_3-z_1^3-z_2^3-z_3^3.
    \]
    Then $p(\Tbf)e_1=4e_8$. Thus $\|p(\Tbf)\|\geq4>\|p\|_\infty$.
    \item[The $L$ operator:] Take $\xi=e_1$. Similarly as above, we compute
    \begin{flalign*}
    L=L_{\Tbf,\xi}=&1\otimes1 +\sum_{i=1}^3z_i\otimes z_i \\
                   &+(z_1z_2-z_3^2)\otimes(z_1z_2-z_3^2)+(z_2z_3-z_1^2)\otimes(z_2z_3-z_1^2)+(z_1z_3-z_2^2)\otimes(z_1z_3-z_2^2)\\
    &+(z_1z_2z_3-z_1^3-z_2^3-z_3^3)\otimes(z_1z_2z_3-z_1^3-z_2^3-z_3^3)\\
    =&P_{\leq 1}+M_{z_1}^\ast p\otimes M_{z_1}^\ast p+M_{z_2}^\ast p\otimes M_{z_2}^\ast p+M_{z_3}^\ast p\otimes M_{z_3}^\ast p+p\otimes p.
    \end{flalign*}
    \item[Interpretation:] Write $L=\sum_{k=0}^3L_k$, where
    \[
    L_0=1\otimes1,\quad L_1=\sum_i z_i\otimes z_i,\quad L_2=\sum_iM_{z_i}^\ast p\otimes M_{z_i}^\ast p,\quad L_3=p\otimes p.
    \]
    It is obvious that $M_{z_i}^\ast L_k M_{z_i}\leq L_{k-1}$ for $k=1, 3$. For $k=2$, note that 
\end{description}
 \[
 M_{z_i}^\ast L_2 M_{z_i}=\sum_{j=1}^3M_{z_iz_j}^\ast p\otimes M_{z_iz_j}^\ast p=L_1.
 \]
 So $L\in\LLL_c^{(3)}(p)$. This implies $\|p\|_\ast\leq 1$. By Lemma \ref{lem: dual SA geq coefficients}, we actually have
 \[
 \|p\|_\ast=1,\quad\text{which implies}\quad\|p\|_{\SArm}\geq\frac{\|p\|_2^2}{\|p\|_\ast}=4>\|p\|_\infty.
 \]
 Using \eqref{eqn:Gamma_star_Gamma}, an alternative way of interpreting $L$ is that
 \[
 L=\Gamma_p^\ast\Gamma_p\big|_{P_{\geq2}}+\max_i\left\|M_{z_i}^\ast\Gamma_p\big|_{P_1}\right\|^2\cdot P_{\leq1}.
 \]
 Thus, this example is also captured by Method 2.

 We also remark that the $L$ operator of the Kaijser-Varopoulos example is  a compression of the $L$ operator
 of the Crabb--Davie example. Indeed, on $P_2$, the $L$ operator of Kaijser--Varopoulos is the compression
 of the $L$ operator of Crabb--Davie to the one dimensional subspace spanned by $q$. Thus, the two counterexamples
 are more closely related than might be apparent at first glance.
\end{exam} 

\begin{exam}[{\bf Holbrook 2001 \cite{Hol01}}]\label{Exam: Hol}
In \cite{Hol01}, Holbrook gave a counterexample of the von Neumann's inequality consisting of three $4\times4$ matrices. Later, Knese proved in \cite{Knese16} that the von Neumann's inequality holds at dimension 3. Thus $4$ is the minimal dimension that the inequality fails. 

\begin{description}
    \item[Original construction:] The Holbrook counterexample is constructed as follows. Let $e, h$ be orthonormal in $\CC^4$, and let $f_1, f_2, f_3$ be unit vectors that span $\{e, h\}^\perp$, such that $\la f_i, f_j\ra=-\frac{1}{2}$ for $i\neq j$. Define the operators $T_1, T_2, T_3$ by 
\[
T_ke=f_k,\quad T_kf_j=\begin{cases}
    h&\text{ if }k=j\\
    -\frac{1}{2}h&\text{ if }k\neq j
\end{cases},\quad T_k h=0.
\]
Let $p=z_1^2+z_2^2+z_3^2-2z_1z_2-2z_2z_3-2z_1z_3$, as in the Varopoulos-Kaijser counterexample. Then $\|p(\Tbf)\|=6=\frac{6}{5}\|p\|_\infty$.
    \item[The $L$ operator:]  We compute
\begin{flalign*}
L_{\Tbf,e}=1\otimes1+\sum_{i=1}^3z_i\otimes z_i-\frac{1}{2}\sum_{i\neq j}z_i\otimes z_j+q\otimes q,
\end{flalign*}
where
\[
q=z_1^2+z_2^2+z_3^2-\frac{1}{2}(z_1z_2+z_2z_3+z_1z_3),
\]
\item[Interpretation:] Again, write $L=L_{\Tbf,e}=L_0+L_1+L_2$, where
\[
L_0=1\otimes 1,\quad L_1=\sum_{i=1}^3z_i\otimes z_i-\frac{1}{2}\sum_{i\neq j}z_i\otimes z_j,\quad L_2=q\otimes q.
\]
We observe that $\|M_{z_i}^\ast L_2 M_{z_i}\|=\|M_{z_i}^\ast q\|_2^2=\frac{3}{2}, i=1, 2, 3$, and $L_1$ is exactly $\frac{3}{2}$ times the projection operator into $\mathrm{span}\{M_{z_i}^\ast q\}_{i=1}^3$. Therefore $M_{z_i}^\ast L_2M_{z_i}\leq L_1$. We verify directly that
\[
M_{z_i}^\ast L_1M_{z_i}=L_0,\quad i=1, 2, 3.
\]
Therefore $L\in\LLL_c^{(2)}(q)$, which implies $\|q\|_\ast\leq1$. Again, by Lemma \ref{lem: dual SA geq coefficients},
\[
\|q\|_\ast=1,\quad\text{which implies}\quad\|p\|_{\SArm}\geq\frac{|\la p, q\ra_2|}{\|q\|_\ast}=6>5=\|p\|_\infty.
\]
As mentioned in Section \ref{sec: methods of constructing L}, this example can also be explained with Methods 3 and  4.
\end{description}
\end{exam}

\begin{exam}[{\bf Dixon 1976 \cite{Dix76}}]\label{exam: Dixon} In Dixon's proof of the lower bound for $C(d,n)$, he developed a method for constructing counterexamples to the von Neumann inequality.
\begin{description}
    \item[Original construction:] Assume $n=2r+1$ is an odd number, and $d>n$. From a combinatorial lemma (\cite[Lemma 3.2]{Dix76}), Dixon constructed a set $\SSS=\{A_j\}_{j=1}^N$ satisfying the following.
    \begin{enumerate}
        \item Each $A_i$ is a subset of $\{1, 2, \cdots, d\}$ containing exactly $n$ elements;
        \item for $j\neq k$, $|A_j\cap A_k|<t=r+1$;
        \item $N\geq{d\choose n}{d\choose n-t}^{-1}{n\choose t}^{-1}$.
    \end{enumerate}
    For $A=\{i_1,\cdots,i_n\}\in\SSS$, write $z^A=z_{i_1}\cdots z_{i_n}$. A result of Kahane, Salem and Zygmund \cite[Chapter 6, Theorem 4]{KahaneBook85}, see also \cite[Lemma 3.1]{Dix76}, ensures that there exist a choice of signs $\{\epsilon_A\}_{A\in\SSS}, \epsilon_A=\pm1$, so that 
    \[
    p(z)=\sum_{A\in\SSS}\epsilon_A z^A,
    \]
    has supremum norm
    \[
    \|p\|_\infty<C\left(dN\log n\right)^{1/2},
    \]
    where $C$ is an absolute constant. Let $\HHH$ be the finite-dimensional Hilbert space with orthonormal basis
    \[
    e(j_1,\cdots,j_m), f(j_1,\cdots,j_m),\quad0\leq m\leq r, 1\leq j_1\leq\cdots\leq j_m\leq d.
    \]
    For $m=0$, simply write as $e$ and $f$. Define
    \[
    T_le(j_1,\cdots,j_m)=e(j_1,\cdots,j_h,l,j_{h+1},\cdots,j_m),\quad 0\leq m\leq r-1,
    \]
    where $h$ is such that $j_h\leq l\leq j_{h+1}$;
    \[
    T_lf(j_1,\cdots,j_m)=\begin{cases}
        f(j_1,\cdots,j_h,j_{h+1},\cdots,j_m)&l=j_h\\
        0&l\notin\{j_1,\cdots,j_m\}
    \end{cases}
    \quad0\leq m\leq r;
    \]
    \[
    T_l e(j_1,\cdots,j_r)={\sum}'\epsilon_{\{l,j_1,\cdots,j_r,i_1,\cdots,i_r\}}f(i_1,\cdots,i_r).
    \]
    Here the summation $\sum'$ is over all integers $1\leq i_1\leq\cdots\leq i_r\leq d.$ One can verify that $\Tbf$ is a commuting tuple of contractions and $p(\Tbf)e=Nf$. Therefore
    \[
    \|p\|_{\SArm}\geq\|p(\Tbf)\|\geq N.
    \]
    Consequently,
    \begin{flalign}\label{eqn: C(d,2r+1) lower bound}
    C(d,2r+1)\geq&\frac{\|p\|_{\SArm}}{\|p\|_\infty}\geq\frac{N}{C\left(dN\log n\right)^{1/2}}=\frac{N^{1/2}}{C(d\log n)^{1/2}}\geq \left(\frac{{d\choose n}{d\choose n-t}^{-1}{n\choose t}^{-1}}{C^2d\log n}\right)^{1/2}\nonumber\\
    =&\left(\frac{(r!)^2(r+1)!}{((2r+1)!)^2C^2\log (2r+1)}\cdot\frac{(d-r)!}{d(d-2r-1)!}\right)^{1/2}\\
    \approx_r& \, d^{r/2}.\nonumber
    \end{flalign}
    \item[The $L$ operator:] By direct computation,
    \[
    \la\Tbf^\alpha e, \Tbf^\beta e\ra=\begin{cases}
        0&|\alpha|\neq|\beta|,\text{ or }\max\{|\alpha|, |\beta|\}>n\\
        \delta_{\alpha,\beta},&|\alpha|=|\beta|\leq r\\
        \sum_\gamma\epsilon_{\gamma\cup\alpha}\epsilon_{\gamma\cup\beta},&|\alpha|=|\beta|>r,
    \end{cases}
    \]
    where in the last case, the summation is taken over all $\gamma$ such that $|\gamma|+|\alpha|=n$, and both $\gamma\cup\alpha$ and $\gamma\cup\beta$ are in $\SSS$.
    From this we can verify that
    \[
    L=L_{\Tbf,e}=\sum_{|\gamma|\leq r}M_{z^\gamma}^\ast p\otimes M_{z^\gamma}^\ast p+P_{\leq r}.
    \]
    \item[Interpretation:] Again, write $L=\sum_{k=0}^nL_k,~L_k\in\BBB(P_k)$. Then it is obvious that $M_{z_i}^\ast L_kM_{z_i}\leq L_{k-1}$ for $k\neq r+1$. Meanwhile,
    \[
    M_{z_i}^\ast L_{r+1}M_{z_i}=\sum_{|\gamma|=r}M_{z_iz^\gamma}^\ast p\otimes M_{z_iz^\gamma}^\ast p.
    \]
    From condition (2) of the set $\SSS$, each $M_{z_iz^\gamma}^\ast p$ is either zero or a monomial, and for different $\gamma$, they are distinct. Therefore $M_{z_i}^\ast L_{r+1}M_{z_i}\leq P_r=L_r$. Also note that $L_n=p\otimes p$. So $L\in\LLL_c^{(n)}(p)$. By the above and Lemma \ref{lem: dual SA geq coefficients},
    \[
    \|p\|_\ast=1,\quad\text{which implies}\quad\|p\|_{\SArm}\geq\frac{\|p\|_2^2}{\|p\|_\ast}=N.
    \]
    The rest of the proof is the same as the original construction.
    Again by \eqref{eqn:Gamma_star_Gamma}, we see that
    \begin{equation*}
      L = \Gamma_p^* \Gamma_p \big|_{\ge r+1} + P_{\le r},
    \end{equation*}
    and $\max_i \| M_{z_i}^* \Gamma_p \big|_{P_r}\| = 1$,
    so Dixon's example can also be explained with Method 2.
    
\end{description}
\begin{rem}\label{rem: Dixon exponential growth}
    If we take $d=kr$ in \eqref{eqn: C(d,2r+1) lower bound}, with $k$ large enough, then the above gives
    \[
    C(kr,2r+1)\gtrsim\frac{1}{r^{1/4}(\log(2r+1))^{1/2}}\left(\frac{(1+\frac{1}{k-2})^{k-1}(k-1)}{2^4}\right)^{r/2}\gtrsim 2^{r/2}=2^{\frac{d}{2k}}.
    \]
    Therefore for some large $k$,
    \[
    C(d)\gtrsim 2^{\frac{d}{2k}}.
    \]
    This shows that $C(d)$ must grow exponentially. In fact, as explained above, the construction of Dixon's examples matches Method 2. Thus the estimates above actually shows that
    \[
    \sup \left\{ \frac{\vertiii{p}_2}{\|p\|_\infty}: p \in P_{d,n}, n \in \mathbb N \right\}\gtrsim 2^{\frac{d}{2k}}.
    \]
  \end{rem}
\end{exam}

\bibliographystyle{plain}
\bibliography{reference}
	
\end{document}